\documentclass[11pt]{article}
\usepackage{fullpage}
\usepackage{amssymb,amsmath,graphicx,amsthm}
\usepackage{color, comment}
\usepackage{graphicx}

\usepackage[utf8]{inputenc}

\usepackage{hyperref}
\hypersetup{
    colorlinks=true,       % false: boxed links; true: colored links
    linkcolor=red,          % color of internal links (change box color with linkbordercolor)
    citecolor=green,        % color of links to bibliography
    filecolor=magenta,      % color of file links
    urlcolor=cyan           % color of external links
}

\newtheorem{thm}{Theorem}[section]
\newtheorem{cor}[thm]{Corollary}
\newtheorem{remark}[thm]{Remark}

\newtheorem{lem}[thm]{Lemma}
\newtheorem{prop}[thm]{Proposition}

\newtheorem*{thm*}{Theorem}

\theoremstyle{definition}
\newtheorem{dfn}[thm]{Definition}
\theoremstyle{remark}

\numberwithin{equation}{section}

\newcommand{\IP}[1]{\left<#1\right>}

\newcommand{\set}[1]{\left\{#1\right\}}

\newcommand{\sr}[1]{\left(#1\right)}

\newcommand{\Integer}{\mathbb{Z}}

\newcommand{\Z}{\Integer}

\newcommand{\Q}{\mathbb{Q}}
\newcommand{\R}{\mathbb{R}}

\newcommand{\eps}{\varepsilon}

\DeclareMathOperator{\E}{\mathbb{E}}     % Without under-subscripts

\DeclareMathOperator{\Ext}{Ext}

\DeclareMathOperator{\diam}{diam}

\renewcommand{\Pr}{}
\let\Pr\relax
\DeclareMathOperator{\Pr}{\mathbb{P}}

\newcommand{\1}[1]{\mathbf{1}_{\set{ #1 } }}

\def\squareforqed{\hbox{\rlap{$\sqcap$}$\sqcup$}}
\def\qed{\ifmmode\squareforqed\else{\unskip\nobreak\hfil
\penalty50\hskip1em\null\nobreak\hfil\squareforqed
\parfillskip=0pt\finalhyphendemerits=0\endgraf}\fi}

\newcommand{\ignore}[1]{ }

\newcommand{\p}{\partial}

\newcommand{\dist}{\mathrm{dist}}
\newcommand{\vphi}{\varphi}

\newcommand{\Ee}{\mathcal{E}}

\newcommand{\Haus}{\mathrm{Haus}}

\newcommand{\cG}{\mathcal{G}}
\newcommand{\cA}{\mathcal{A}}
\newcommand{\cB}{\mathcal{B}}
\newcommand{\cC}{\mathcal{C}}
\newcommand{\cD}{\mathcal{D}}

\def\cE{\mathcal{E}}
\def\cF{\mathcal{F}}

\def\cL{\mathcal{L}}
\def\cM{\mathcal{M}}
\def\cR{\mathcal{R}}
\def\cS{\mathcal{S}}
\def\cX{\mathcal{X}}

\DeclareMathOperator{\tmix}{\textit{t}_{mix}}

\DeclareMathOperator{\Reff}{\cR_{\text{eff}}}

\newcommand{\ip}[1]{\langle #1 \rangle}

\def\pd{\partial}

\usepackage[usenames,dvipsnames]{pstricks}
 \usepackage{epsfig}
 \usepackage{pst-grad} % For gradients
 \usepackage{pst-plot} % For axes

\begin{document}

\title{Condensation of a self-attracting random walk}
\author{by Nathana\"el Berestycki\thanks{University of Cambridge. Supported in part by EPSRC grants EP/GO55068/1 and EP/I03372X/1} \and Ariel Yadin\thanks{Ben Gurion University of the Negev.}}

\maketitle

\begin{abstract}
We introduce a Gibbs measure on nearest-neighbour paths of length $t$ in the Euclidean $d$-dimensional lattice, where each path is penalised by a factor proportional to the size of its boundary and an inverse temperature $\beta$.
We prove that, for all $\beta>0$, the random walk condensates to a set of diameter $(t/\beta)^{1/3}$ in dimension $d=2$, up to a multiplicative constant. In all dimensions $d\ge 3$, we also prove that the volume is bounded above by $(t/\beta)^{d/(d+1)}$ and the diameter is bounded below by $(t/\beta)^{1/(d+1)}$.
 Similar results hold for a random walk conditioned to have local time greater than $\beta$ everywhere in its range when $\beta$ is larger than some explicit constant, which in dimension two is the logarithm of the connective constant.
\end{abstract}

 \medskip \noindent {\small \centerline{\textbf{Résumé.}}} 
 
{\small Nous introduisons une mesure de Gibbs sur les chemins de longueur $t$ dans le réseau Euclidien de dimension $d$, telle qu'un chemin donné est penalisé par un facteur proportionnel à la taille de sa frontière et l'inverse d'une température $\beta>0$. Nous montrons qu'en dimension $d=2$, la marche aléatoire se condense dans un ensemble de diamètre $(t/\beta)^{1/3}$ à une constante multiplicative près. En dimensions $d\ge 3$, nous montrons que la marche occupe un volume inférieur à $(t/\beta)^{d/(d+1)}$ et son diamètre est au moins $(t/\beta)^{1/(d+1)}$. Des résultats similaires sont obtenus pour une marche aléatoire conditionnée à avoir un temps local supérieur à $\beta$ en chaque point visité, pourvu que $\beta$ soit supérieur à une constante explicite qui en deux dimensions est égale au logarithme de la constante de connectivité. }



\begin{figure}[h]\label{fig:simul}\centering
\includegraphics[scale=.3]{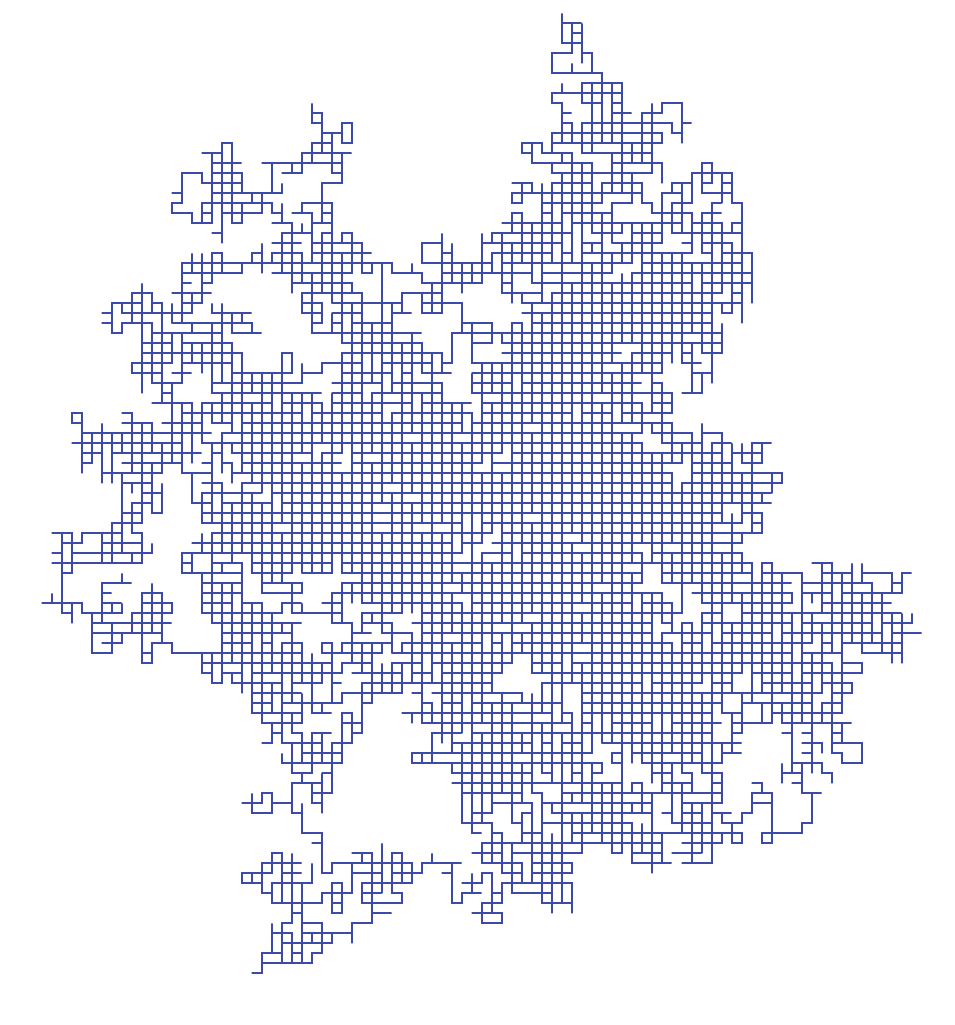}
\includegraphics[scale=.3]{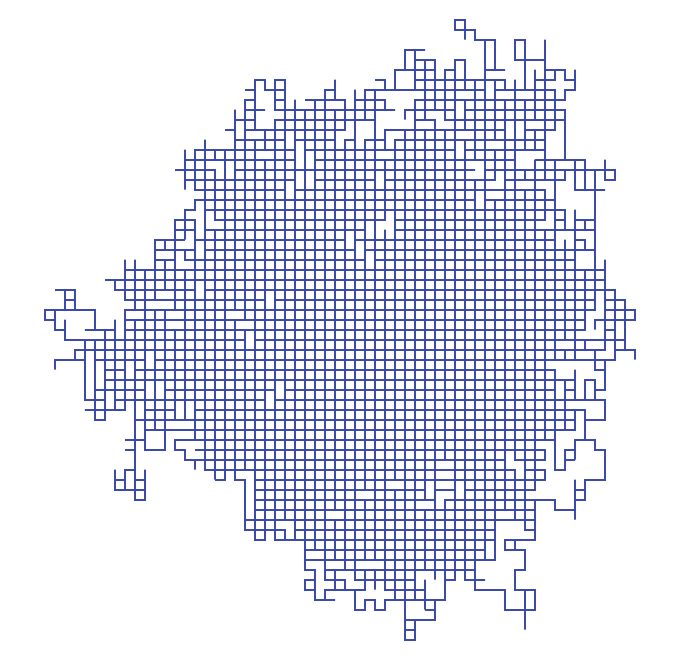}
\includegraphics[scale=.3]{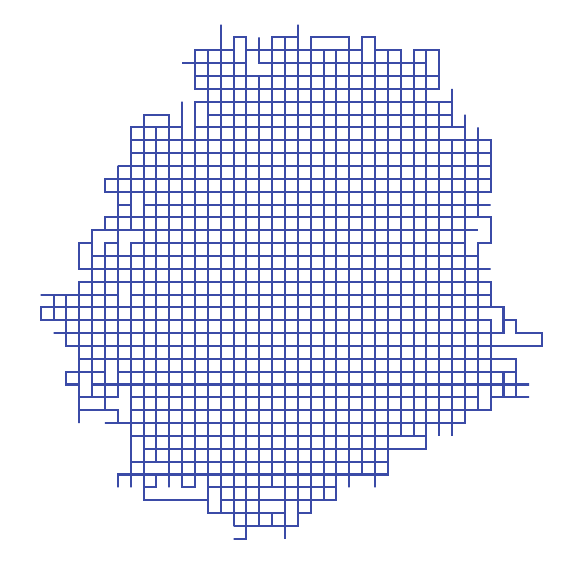}
\includegraphics[scale=.3]{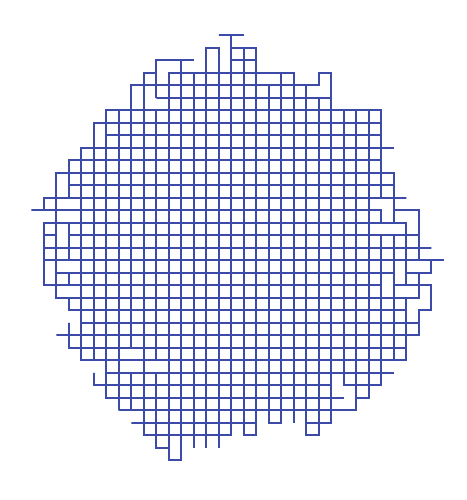}

\caption{Simulations using a Gibbs sampler algorithm of a random walk cluster with $t = 25,\!000$ steps, corresponding to $ \beta = 0.01$, $\beta = 0.1$, $\beta = 1$, and $\beta = 2$.}
\end{figure}

\tableofcontents

\section{Introduction}

\subsection{Statement of the main results}

Let $d\ge 2$ and let $\Omega$ be the space of nearest neighbour, right continuous, infinite paths $(\omega_t,t\in[0, \infty))$ on $\mathbb{Z}^d$, and let $(X_t(\omega), t\ge 0)$ be the canonical process. For $x \in \Z^d$, let $\Pr_x$ denote the law of simple random walk on $(\Omega, \cF)$ in continuous time where each edge has rate one started from $x$, and where $\cF$ denotes the $\sigma$-field generated by $X$. We call $\Pr = \Pr_0$.

Our main result deals with geometric properties of some penalisations of random walks on $\mathbb{Z}^d$ by their boundary. More precisely, we introduce a Gibbs measure $\mu = \mu_t$ on random paths defined as follows. Let $R_t = \{v \in \Z^d: X_s = v \text{ for some } s \le t\}$ denote the \emph{range} of the walk at time $t$. For a given time $t$, we consider the Hamiltonian $H$ given by
\begin{equation}\label{defH}
H(\omega) = |\p R_t|,
\end{equation}
where for a set $G$, $\pd G = \{ x \in G: x \sim y \text{ for some } y \notin G\}$ is the (inner) vertex-boundary of $G$. (Here $x\sim y$ means that $x$ is a neighbor of $y$ in $\Z^d$). The associated Gibbs measure on random paths $\mu = \mu_t$ is obtained by considering the measure $\mu$ defined by
\begin{equation}\label{gibbs}
\frac{d\mu}{d\Pr}(\omega) = \frac1{Z}\exp( - \beta H(\omega))
\end{equation}
on $\cF$.
Here $\beta >0$ is a positive number playing the role of inverse temperature and $Z = Z(t,\beta) = \E( \exp( - \beta |\p R_t|)) $ is a normalising factor called the partition function. In plain words, the Gibbs measure $\mu$ penalises every site on the boundary of the range $R_t$ by a fixed amount $e^{-\beta}$. Hence $\mu$ favours ``highly condensed" configurations. Interpreting the random walk $(X_0,\ldots, X_t)$ as a chain of monomers, the Gibbs measure $\mu$ describes the law of a diluted polymer in a poor solvent.

\medskip We will be interested in describing the geometry of $R_t$. With a constraint on its maximal volume and a penalty in case of a large boundary, this problem is closely related to the question of \textbf{phase separation} in statistical mechanics. This is a classical topic with a long and distinguished history, for which we mention only a few major milestones. Traditionally, microscopic models for this phenomenon have been based on the framework of either percolation or the Ising model. Either way, a key goal is to prove a shape theorem for the cluster. Such a result can then be viewed as a microscopic justification for the Wulff  construction (\cite{Wulff}), which is a method to determine the equilibrium shape of crystals based on surface energy minimisation.  In the percolation context, a rigorous derivation of the limiting shape was first given by Alexander, Chayes and Chayes \cite{wulff1} in two dimensions, while in the context of the Ising model, this was achieved slightly earlier in a celebrated work of Dobrushin, Koteck\'y and Shlosman \cite{DKS} (for which preliminary announcements can be found in \cite{DKSprelim1, K}, as noted by an anonymous referee).
This result was derived again by Pfister \cite{Pfister}, see also the papers by Ioffe and Schonmann \cite{IoffeSchonmann} extending the results of \cite{DKS} to all subcritical temperatures. The three-dimensional case, which is the most delicate, was handled only relatively recently by Cerf \cite{cerf}, with earlier work by Bodineau \cite{Bodineau}, Cerf and Pisztora \cite{CerfPistora} as well as Bodineau, Ioffe and Velenik \cite{BodineauIoffeVelenik}. See \cite{MinlosSinai} for an early reference on the problem of phase separation in the context of the Ising model,  \cite{cerf} for a recent monograph giving a detailed overview of the subject.

As mentioned above, until now the question of phase separation has been studied rigorously mostly in the context of percolation and the Ising model. As far as we are aware, the present paper is the first attempt to study the question through genuinely $d$-dimensional random walks. Note however that in the 1-dimensional SOS model, namely for a 1+1-dimensional random walk conditioned on describing an atypically large arithmetic area, it was proven by Dobrushin and Hryniv in  \cite{DobrushinHryniv} that a limiting Wulff shape arises.

Our main result gives precise estimates for the condensation effect that results from the attractive self-interaction in dimension $d\ge 2$. If $G \subset \Z^d$, let $\diam G$ denote the (Euclidean) diameter of $G$:
$$
\diam G = \sup\{ |z - w| \ : \  z,w \in G\}.
$$
While we are currently unable to derive a shape theorem, our main result suggests that the limit shape, if it exists, has diameter of order $t^{1/(d+1)}$.

\begin{thm}\label{T:2dimGibbs}
  \begin{enumerate}
  \item Let $d= 2$. For any $\beta_0>0$ there exist positive constants $c_1, c_2$ depending only on $\beta_0$ such that for all $\beta > \beta_0$,
  \begin{equation}\label{Teq:2dimGibbs}
 \mu_t \left( c_1 \left( \frac{t}{\beta}\right)^{1/3} \le \diam(R_t) \le c_2 \left( \frac{t}{\beta}\right)^{1/3}  \right) \to 1 ,
  \end{equation}
 as $t \to \infty$.

 \item For all $d\ge 3$ we have for all $\beta_0$ and for all $\beta > \beta_0$,
 \begin{equation}
\mu_t \left( \diam (R_t) \ge c_1 \left(\frac{t}{\beta}\right)^{1/(d+1)}
 \qquad \text{ and } \qquad |R_t| \le c_2  \left(\frac{t}{\beta}\right)^{d/(d+1)} \right) \to 1
 \end{equation}
 as $t\to \infty$, where
 the constants $c_1$ and $c_2$ depend only on $d$ and $\beta_0$.
 \end{enumerate}
\end{thm}

In principle, the inverse temperature $\beta$ may even be chosen to depend on $t$, in which case the same result holds if we also assume $\beta = \beta(t)$ satisfies $\beta(t)/t \to 0$ as $t \to \infty$.

\medskip Note that the proof gives precise estimates on the probability of these events, as well as  estimates on the partition function $Z = Z(t,\beta)$. We refer the reader to Theorem \ref{T:2dimGibbs+} for a precise statement.
In dimensions $d\ge 3$, we conjecture that $\diam(R_t)$ scales like $(t/\beta)^{1/(d+1)}$, but we have only obtained a lower bound. Our proof supports this conjecture, but the upper bound is elusive (roughly because of topological complications in dimensions $\ge 3$). Nevertheless we still manage to get an upper bound on the volume which is consistent with this conjecture. This difficulty is a common feature of all works on Wulff crystal.

\medskip A related problem was studied by E. Bolthausen \cite{bolthausen} in dimension $d=2$. In that work, the energy $H(\omega)$ serving to define the Gibbs measure $\mu$ in \eqref{gibbs} is taken to be  $\hat H(\omega)  = |R_t(\omega)|$, the size of the range (as opposed to that of its boundary). Thus $d\hat \mu = \hat Z_t^{-1} \exp( - \beta\hat H(\omega) ) d \Pr$.
 Bolthausen's result is that the random walk condensates to a set of diameter $t^{1/4}$, which is close in the Hausdorff sense to a Euclidean ball of that diameter. In both problems, good bounds on the partition functions $Z_t$ and $\hat Z_t$ play a crucial role. In the case where the energy is just the volume,  we have $\hat Z_t = \E( \exp ( - \beta |R_t|))$, and precise asymptotics for this quantity were already obtained by Donsker and Varadhan \cite{DVrw}. This is a considerably easier problem than the one considered here, essentially because $|R_t(\omega)|$ is ``almost" a continuous function of its local time profile, viewed as a probability measure on $\Z^d$. In particular, the powerful machinery of large deviations theory provides the right tools to study that question. This goes a long way in explaining the appearance of the Euclidean ball as a limit shape, and explains why the inverse temperature $\beta$ is not a relevant parameter in that model.

\medskip  In contrast, here we believe that the limit shape depends on $\beta$ and is not a rotationally symmetric ball.
So the microscopic geometry of the lattice is important even to determine the macroscopic shape of the random walk cluster, and thus there is no hope in directly applying the Donsker--Varadhan large deviations machinery to the problem. A related major difficulty is that two local time profiles can be macroscopically close in the $\ell^1$ sense, say, even though the sizes of their boundaries are of widely different orders of magnitude.
%To get around this difficulty, we perform a suitable change of measure and obtain quantitative estimates to derive the lower bound on the partition function. This is the main technical part of the paper; the key is to understand how a random walk gets a small (or smooth) boundary.  Once this is done, we rely on a discrete isoperimetric inequality to deal efficiently with the entropy coming from the large number of possible shapes of the range.

\subsection{Some related problems}

Our technique is sufficiently robust that it yields similar results for a number of models which turn out to be quite closely related.
One interesting case is the following conditioning problem, initially suggested by Itai Benjamini in private communication with the first author (in fact, it was this question which was initially the focus of the present investigation). For $t\ge 0$, define the event $\cE_t$ as follows:
\begin{equation}
  \label{D:event}
  \cE_t = \{L(t,x) \ge \beta, \forall \ x \in R_t\}.
\end{equation}
Here $L(t,x) = \int_0^t 1_{\{ X_s = x \}} ds$ is the amount of time the walk spends at a vertex $x$.
Conditioning on this event gives a uniform lower bound on the density of local time uniformly over the range $R_t$. This also favours highly condensed configurations. In Benjamini's original question, it was assumed that $\beta =\beta(t) \to \infty$ as $t\to \infty$. The question was to decide whether, conditional on the event $\cE_t$, there is a shape theorem for the range $R_t$.

As we will see later on (see for instance \eqref{expcostEt}), the conditioning also heavily penalises shapes $R_t$ with large boundaries: essentially, every point on the boundary penalises the shape by a factor of order $e^{- \beta}$, and so a behaviour similar to Theorem \ref{T:2dimGibbs} may be expected. In particular, the conditioning is already highly nontrivial when $\beta(t) \equiv \beta$ is a fixed constant. This is perhaps counterintuitive initially, since in dimension $d=2$ for instance, typical points are visited logarithmically many times, so the constraint $\cE_t$ does not seem ``very" singular.

Unlike in Theorem \ref{T:2dimGibbs}, we will need an assumption that $\beta>\beta_0$, where $\beta_0$ is an explicit constant: $\beta_0 = \log \alpha$, where $\alpha$ is the connective constant of $\Z^2$.  (In dimension $d\ge 3$, that constant takes a different value related to a notion of self-avoiding surfaces, see \eqref{beta0} for the definition).

\begin{thm}\label{T:Cond2dim}
\begin{enumerate}
\begin{comment}
\item For any $\beta_1>0$ and all $\beta>\beta_1$, the partition function admits the lower bound
$$
Z(t,\beta) \ge
\exp\left( - \gamma  t^{1- 2/(d+1)} \beta^{2 / (d+1)} \right) ,
$$
where $\gamma$ is a constant depending only on $d\ge 2$ and $\beta_1$.
\end{comment}
\item
 Let $d= 2$. Let $\beta_0 = \log \alpha$, where $\alpha$ is the connective constant. For all $\beta > \beta_1 > \beta_0$, we have
  \begin{equation}\label{Teq:Cond2dim}
  \Pr \left( \left. c_1  \left( \tfrac{t}{\beta}\right)^{1/3}  \le \diam (R_t) \le c_2 \left( \tfrac{t}{\beta}\right)^{1/3}  \right| \cE_t \right) \to 1
  \end{equation}
as $t \to \infty$, where the positive constants $c_1$ and $c_2$ depend only on $\beta_1$.

 \item For all $d\ge 2$, there exists $\beta_0 = \beta_0(d)$ such that for all $\beta>\beta_1 > \beta_0$,
 \begin{equation}
\Pr \left( \left. \diam (R_t) \ge c_1 \left(\tfrac{t}{\beta}\right)^{1/(d+1)} ; \text{ and } |R_t| \le c_2  \left(\tfrac{t}{\beta}\right)^{d/(d+1)} \right| \cE_t \right) \to 1
 \end{equation}
 as $t\to \infty$, where the positive constants $c_1$ and $c_2$ depend only on $d$ and $\beta_1$.
 \end{enumerate}
\end{thm}

As in Theorem \ref{T:2dimGibbs}, the result remains valid if $\beta = \beta(t)$ is allowed to depend on $t$, provided also that $\beta(t) / t \to 0$ as $t \to \infty$.

Another variant consists in taking a slightly different Hamiltonian $\tilde H$, defined by
\begin{equation}
\tilde H(\omega) = \sum_{x \in \p R_t} L(t,x),\label{Hvariant}
\end{equation}
Thus the penalisation takes into account not only the size of the boundary, but also the amount of time spent on it.
 Define $d\tilde \mu_t = (\tilde Z)^{-1} \exp( - \beta \tilde H) d\Pr $ on $\cF $.
\begin{thm}\label{Tvariant}
Theorem \ref{T:2dimGibbs} still holds true with $\tilde \mu_t$ instead of $\mu_t$.
\end{thm}

\paragraph{Acknowledgements.} This work started when AY was a Herschel Smith postdoctoral fellow in 2009--2010 at the Statistical Laboratory, University of Cambridge. We gratefully acknowledge the financial support of the Herschel Smith fund and EPSRC grant EP/GO55068/1 as well as EP/L018896/1 and EP/I03372X/1.
The first author is grateful for the hospitality of the Theory Group at Microsoft Research, where part of this work was carried out. He would also like to thank Omer Angel, Ori Gurel--Gurevitch, Yuval Peres and Ofer Zeitouni for useful conversations, and Tom Begley for the pictures. We are very grateful to two anonymous referees for their comments which improved the presentation and pointed out some mistakes in earlier versions of the paper.

\paragraph{Updates.} Since the first version was posted to the arXiv in 2013, there has been some progress on questions inspired by this paper. For instance, the series of works by Asselah and Schapira \cite{AS1, AS2} discusses a large deviation principle for the boundary of the range of a simple random walk in dimensions $d \ge 3$. In a different direction, a series of two articles by Biskup and Procaccia \cite{BP1, BP2} study a direct analogue of \eqref{gibbs} in the two-dimensional case, except that the  boundary of the range is understood to mean the edge boundary (whereas we consider here the vertex boundary) and the weight of the edges is also allowed to be random. Then by letting $t\to \infty$ and then $\beta \to \infty$ they are able to obtain a limit theorem for the shape of the range, which is nonrandom but depends on the law of the weights. In the particular case of deterministic (nonrandom) edge weights, this limit shape is simply the unit square. This is the analogue of our conjecture here (see Section \ref{S:opb}) that the limit shape is a diamond when $\beta \to \infty$. The difference between their square and our diamond comes from the difference between  edge and vertex boundary in the formulation of the problem in \cite{BP1, BP2}.

\subsection{Main ideas in the proof; organisation of the paper}

Since the proof of Theorem \ref{T:2dimGibbs} has many technical aspects, and involves plenty of careful computations,
let us provide a sketch of the main ideas involved.
Recall that we are trying to estimate the radius and volume
of the trace of the random walk penalised by its local time at
its boundary. A very rough heuristic for the size of
the diameter is as follows. The probability of staying in
a box $S$ of diameter $L$ is approximately
$\exp( - \text{const.} \times t / L^2)$, since the random walk has a probability of order 1 to escape $S$ every $L^2$ units of time. In this case, one can expect
the size of the boundary to then be approximately $L^{d-1}$,
so the corresponding energy of such a configuration is of
order $\exp ( - \beta L^{d-1})$. Balancing energy and
entropy gives us $\beta L^{d-1} = t / L^2$ and so
$L = (t / \beta)^{1/(d+1)}$, which is indeed the conjectured order of the diameter in all dimensions $d \ge 2$ (see Theorem \ref{T:2dimGibbs}).

\medskip The main issue in translating this rough heuristic to a rigorous argument is that the random walk could stay in a box of size $L$ while having a boundary much larger than $L^{d-1}$: this will be the case if the boundary is in some sense rough or fractal, which is \emph{a priori} the case at least in small dimensions (recall that in dimension $d=2$, the dimension of the outer boundary of Brownian motion is $4/3$). This raises serious questions about the heuristic argument above: could the probability of staying in a box of size $L$ \emph{and} have a smooth boundary be  substantially smaller than $\exp (- \text{const.} \times t / L^2)$?
Fortunately we answer by the negative. Correspondingly, our main task is to prove a lower bound on the partition function $Z$ (see Proposition \ref{P:lbZ}), which establishes one scenario of probability roughly $\exp ( - \text{const.} \times t / L^2)$ where the boundary of the range is of size approximately $L^{d-1}$. From this point of view, the most delicate case appears to be $d=2$; yet surprisingly this is where our results are also the most precise.

\medskip In order to do this, we first have to guess the profile of local times $(\pi(x))_{x \in S}$, in the box $S$ of diameter $L$,
achieved by a random walk conditioned to have a small boundary. The trickiest part is to guess the behaviour of this profile close (at micro- and mesoscopic distance) to the boundary of the box. We define a specific profile which with hindsight should be almost the optimal one.
We then have to compute the cost of achieving this profile, and show that the boundary has the desired size $O(L^{d-1})$ with this profile.

\medskip This leads us to a change of measure
argument, as done in \cite{bolthausen}, and we must estimate the Radon--Nikodym derivative under the tilted measure. The main term turns out to be $\exp (  \int_0^t \frac{\Delta f}{f}(X_s) ds )$, where $f(x) $ is the square root of the local time profile $\pi$ which we seek to impose (see Lemma \ref{L:PQ}). If the local times of $X$ are well approximated by the profile $\pi$ then it is relatively easy to conclude (using careful second-order Taylor expansions, see  Lemma \ref{lem:RN bound}) that this Radon--Nikodym derivative is indeed of order $\exp ( - \text{const.} \times t /L^2)$, as desired. Hence what is needed is a precise control of the large deviations of local time at points under the tilted measure. This is achieved by a careful analysis done in Lemma \ref{L:LTconc},
and our main use of it is summarised in Corollary \ref{cor:tail}. Roughly speaking, to obtain good large deviation control on the local time at a point $x$ which might be close to the boundary of $S$, it suffices to show that there is a positive chance to hit the point $x$, every $1/\pi(x)$ units of time. This is achieved through a quantitative analysis of the tilted measure, using electrical network theory, and is the main purpose of Section \ref{scn:quant est}. This is one of the most technical parts of the paper, and is particularly delicate in the case $d=2$ (reflecting the above mentioned difficulty).

%After obtaining large deviations for all appropriate quantities above, these are combined
%with careful second order Taylor expansions of the proposed profile $\pi$, see Lemma \ref{lem:RN bound}.
%This culminates in the bound on the partition function in Proposition \ref{P:lbZ}.

\medskip Finally, in  Section \ref{S:proofs},
we apply the bound on the partition function bound to control the radius and volume of the
penalised random walk trace. This is done mainly using discrete isoperimetric inequalities.

\medskip Let us note that our methods work even for the {\em constant} $\beta$ regime
(not just $\beta(t) \to \infty$).
In order to achieve this, it was necessary to correctly pick an accurate local time profile $\pi$
for the lower bound on the partition function $Z(t,\beta)$.
The naive choice in this case (essentially the normalised squared principal eigenfunction) was not good enough for constant $\beta$ because of how it behaves near the boundary.
We discuss the required properties of this profile in the beginning of Section \ref{scn:quant est},
adjacent to the definition of $\pi$.
(For example, the polylogarithmic terms
appearing in the definition of $\pi$ are essential
for the analysis to work, but would be absent in the naive choice of $\pi$.)

\section{Quantitative estimates for lower bound on $Z$}
\label{scn:quant est}

\subsection{Change of measure}

As mentioned above, a main technical part of this paper consists in deriving good lower bounds on the partition function $Z$, which hold in all dimensions $d\ge 2$. In order to do this,
we introduce a change of measure which is key to our analysis. This is a relatively standard technique in large deviations (see, e.g., Bolthausen's article \cite{bolthausen}
as well as the work of G\"artner and den Hollander \cite{GH} on intermittency of parabolic Anderson model). However, the precise change of measure which needs to be performed here is much more delicate than usual. The analysis of the titled walk in particular will require a host of tools from the quantitative theory of Markov chains: we will need very precise information about how the tilted walk behaves at microscopic and mesoscopic distances away from the boundary $\partial S$.

Let $ \pi \in \cM_1(\Z^d)$ be a probability measure on $\Z^d$,
and define a law $\Q$ on $(\Omega, \cF)$ as the Markov chain on $\Z^d$ having the transition rates $Q(x,y) = \sqrt{\pi(y) / \pi(x)}$ for $x\sim y$ and $\pi(x)>0$ or, equivalently, infinitesimal generator defined by
\begin{equation}
Qf(x) = \sum_{y \sim x} \sqrt{\frac{\pi(y)}{\pi(x)}} [f(y) - f(x) ].
\end{equation}
It is immediate (but essential) that $\pi$ is a reversible equilibrium measure under $\Q$. We will also let $\Q_x$ denote the law of this Markov chain started from a given vertex $x \in \Z^d$.
The choice of $\pi$ will be crucial to our proof. Let $L$ be least integer greater than $(t/\beta)^{1/(d+1)}$, i.e., $L = {\lceil (t/\beta)^{1/(d+1)}\rceil}$. Let $S = [-L,L]^d \cap \Z^d$ be the cube of side length $2L$.
Our choice of $\pi$ is extremely delicate and is determined by the following requirements.

\begin{itemize}

\item $\pi$ must be chosen so that the walk never leaves the cube $S$, and should spend most of its time in the ``bulk" of the cube $S$, so $\pi(x) \asymp 1/L^d$ near the centre.

\item $\pi$ must be chosen so that by time $t$, points on the boundary $\partial S$ are visited, but typically only a finite (Poisson-like) number of times, so $\pi(x) \asymp 1/t$ near the boundary, i.e., $\pi(x) \asymp 1/ L^{d+1}$. At a finite but large distance from the boundary, the mean number of visits should still be finite but large.

\item $\pi$ must be a ``reasonably smooth" function near the boundary, so that achieving the profile $\pi$ is not too unlikely (we are aiming for  probability of order $\exp ( - c t /L^2)$, which is roughly the probability of staying in a cube of size $L$ for time $t$).

\end{itemize}

These three conditions would ensure that the boundary of the range is not much bigger than the boundary of the cube $S$, while the smoothness condition ensures that $\pi$ is not too unlikely. Recall in particular that $L$ was chosen so that the entropic cost, $\exp( - ct/L^2)$, balances the energy cost $\exp ( - c \beta L^{d-1})$.

In view of the above requirements it might be natural to take $\pi(x) \asymp \dist(x, \partial S) / L^{d+1}$, i.e., increases linearly with the distance to the boundary of $S$. While this clearly fulfils the first and second point, it turns out that the Dirichlet energy of $\sqrt{\pi}$ (which ends up governing how likely it is to achieve $\pi$) is too high by a logarithmic factor. Instead, the specific choice of $\pi$ is as follows. For $0\le r \le L $, let
$$S_r = \{z \in S \ : \ \dist(z, \p S) = r\} = \{ z \in S \ : \ \| z\|_\infty = L-r \},$$
where $\dist(\cdot, \cdot)$ refers to the graph distance on $\Z^d$ and for a point $z = (z_1, \ldots, z_d) \in \Z^d$, $\|z\|_\infty = \max_{1\le i \le d} |z_i|$. Then, for $x \in S_r$, set
\begin{equation} \label{eqn:2D trans prob}
 \pi(x) =C_{\eqref{eqn:2D trans prob}} \mu_r  \text{ where } \mu_r: =
 \begin{cases}
 \frac{r+1}{L^{d+1} (\log (r+2))^2} & \text{ if } r \le L/2\\
\left( \sqrt{\mu_{L/2}} +  \frac{(r - L/2)}{L^{(d+2)/2}}\right)^2 & \text{ if } r \ge L/2 .
\end{cases}
\end{equation}
Let $ \pi(z)=0$ for $z \not\in S$, and the constant $C_{\eqref{eqn:2D trans prob}}$ is chosen so that $\sum_z \pi(z) = 1$.
It can then be checked that $C_{\eqref{eqn:2D trans prob}}$ is uniformly bounded away from 0 and infinity and converges to a constant as $L$ tends to infinity.

We comment briefly on the choice of $\pi$. In view of large deviation theory and the Donsker--Varadhan principle, the most natural choice \emph{a priori} is to take $\pi$ to be the square of the first Dirichlet eigenfunction on $S$, normalised to have unit mass. This is for instance what is used in Bolthausen's work \cite{bolthausen} with some additional tweaking near the boundary of the shape (see the definition of $\tilde \psi$ on p.893 of \cite{bolthausen}). However this turns out to be ``too flat" near the boundary, making the second requirement untrue.

Our choice means that the growth of $\pi$ is much steeper near the boundary. The slightly sublinear growth of $\pi$ near the boundary, in $r/(\log r)^2$, is in fact the crucial feature of this choice: the linear factor $r$ guarantees that points at a large distance from the boundary have a large mean number of visits, while the correcting factor in $1/(\log r)^2$ ensures that $\pi$ is smooth enough that achieving a profile $\pi$ has a probability of the right order of magnitude.

\paragraph{Orientation.} At the technical level, we recall that our argument is organised as follows. Roughly speaking, we wish to obtain large deviation bounds on the local time accumulated at a point $y \in S$ under the tilted measure $\Q$ (Lemma \ref{L:LTconc}). The key for doing so will be to show that $y$ is hit sufficiently frequently, and in particular to obtain exponential tails on the hitting time of $y$ (Proposition \ref{P:exphit}, using electrical network theory). Once Lemma \ref{L:LTconc} is proved, we use the concentration of local time to estimate the Radon-Nikodym derivative of $\Pr$ with respect to $\Q$ (Lemma \ref{lem:RN bound}) and hence estimate the partition function $Z(t, \beta)$ (Proposition \ref{P:lbZ}).

\subsection{Crude estimate on mixing time}

\medskip Our first goal is to prove a crude bound on the mixing time of the Markov chain defined by $\Q_x$, which is needed at various points in our argument. We do this by estimating the spectral gap of the Markov chain, using the method of \emph{canonical paths} of Diaconis and Saloff-Coste. We use the standard canonical paths on $\Z^d$: that is, for $x, y \in S$, we define the path  $\gamma_{x,y}$ as follows.
We first try to match the first coordinate of $x$ and $y$, then the second coordinate, and so on
until the last coordinate. Each time, the change in coordinate is monotone. As an example
if $d = 2$ and $x = (x_1; x_2)$ and $y = (y_1; y_2)$, let  $z = (y_1; x_2)$. Then
$\gamma_{x,y}$ is the union
of two straight segments,
going horizontally from $x$ to $z$ and
then vertically from
 $z$ to $y$. We call $|\gamma|$ the length (number of edges) of a path $\gamma$. If $e = (x,y)$ is an edge, let $q(e) = \pi(x) Q(x,y)$ be the equilibrium flow through $e$.

\begin{lem}
\label{Poinc} Let $E$ denote the set of edges within $S$.
$$
B = \max_{e \in E} \left\{ \frac1{q(e)} \sum_{x,y: e \in \gamma_{x,y}} |\gamma_{x,y}| \pi(x) \pi(y)\right\}.
$$
Then $B  \le C_{\ref{Poinc}} L^2$ for some constant $C_{\ref{Poinc}}>0$.
\end{lem}

\begin{proof}
Fix an edge $e$ and suppose $\dist(e, \p S) = r$. Say that a point $x$ is below $e$ if $\dist(x, \p S)\le r$, and otherwise say that $x$ is above $e$. Note that if $e \in \gamma_{x,y}$, $x$ and $y$ cannot be both above $e$.
Indeed, if $m_i = \min \set{x_i,y_i}$ and $M_i = \max \set{x_i , y_i}$ then
$\gamma_{x,y} \subset \prod_i [m_i,M_i] \subset S$, and $x,y$ are two corners of this hypercube.
So any point on $\gamma_{x,y}$ must be further from $\p S$ than one of $x$ or $y$.

Therefore at least one of $x$ or $y$ is below $e$, say $x$. In this case $\pi(x) /q(e) \le O(1)$. Moreover it is elementary to check that the number of pairs of points $x, y \in S$ such that $e \in \gamma_{x,y}$ is at most $O( L^{d+1})$. Indeed, suppose that the two endpoints of $e$ differ only in the $i$th coordinate with $1\le i \le d$. Then the coordinates $1, \ldots, i$ of $x$ can be chosen arbitrarily among $O(L)$ possibilities (while the remaining coordinates are fixed and imposed by those of either endpoint of $e$). Conversely, the coordinates $i, \ldots, d$ can be chosen arbitrarily among $O(L)$ choices for $y$, and the remaining coordinates are fixed and imposed by those of either endpoint of $e$. Consequently, the total number of choices for $x$ and $y$ such that $e \in \gamma_{x,y}$ is at most $O(L^i) \times O( L^{d-i +1}) = O( L^{d+1})$.

Therefore, using the facts that $\pi(y) \le O(1/L^d)$ and $|\gamma_{x,y}| = O(L)$,
\begin{align*}
\frac1{q(e)} \sum_{x,y : e \in \gamma_{x,y}} |\gamma_{x,y}| \pi(x) \pi(y)
%& \le C L  \sum_{x,y : e \in  \gamma_{x,y}} \pi(y) \frac{\pi(x)}{q(e)} \\
& \le C L  \frac1{L^{d}} \#\{ x,y \in S: e \in \gamma_{x,y}\}  \le CL^2
\end{align*}
as desired.
\end{proof}

By Theorem 3.2.1 in \cite{lsc}, it follows that if $\texttt{gap}$ is the spectral gap of the Markov chain, then $\texttt{gap} \ge 1/({C_{\ref{Poinc}}L^2})$. (In fact, that result holds for discrete time chains but it is straightforward to adapt the proof to the continuous time case).
Now, it is well known that estimates on the spectral gaps yield estimates on the heat kernel. More precisely,
$$
|\Q_x(X_t = y ) - \pi(y)| \le \sqrt{\pi(y) / \pi(x)} e^{ - \texttt{gap } t}.
$$
(See,  {\em e.g.}, the proof of Corollary 2.1.5 of \cite{lsc}). Let
$$
\tmix = \inf\left\{ t \ge 0 \ : \ \text{ for all } x,y \in S: \left|\frac{\Q_x(X_t = y ) }{ \pi(y) }-1 \right| \le 1/2 \right\}
$$
From Lemma \ref{Poinc} we deduce that for all $x,y \in S$,
$$
|\Q_x(X_t = y ) - \pi(y)| \le \sqrt{L} e^{ -  t /(C_{\ref{Poinc}} L^2)}.
$$
Since $\pi(y) \ge c/ L^{d+1}$ for all $y \in S$, it follows that taking $t \ge CL^2 \log L$ with some sufficiently large constant $C$,
$$
|\Q_x(X_t = y ) - \pi(y)| \le \tfrac12 \pi(y).
$$
Thus we have proved:
\begin{equation}\label{tmix}
\tmix \le C_{\ref{tmix}}L^2 \log L
\end{equation}

\begin{comment}
In fact, the same argument also shows that if $(\tilde X_n)_n$ is the jump chain of $X$ and
$$\widetilde \tmix = \inf\left\{ n \ge 0 \ : \ \text{ for all } x,y \in S: \left|\frac{\Q_x(\tilde X_n = y ) }{ \tilde \pi(y) }-1 \right| \le 1/2 \right\}$$ where $\tilde \pi$ is the invariant measure of $\tilde X$, then
\begin{equation}\label{tmixtilde}
\widetilde \tmix \le C_{\ref{tmixtilde}}L^2 \log L
\end{equation}
\end{comment}

\subsection{Flows and hitting estimates}

In this section we start deriving a key estimate used in the proof, which gives exponential decay of the tail for the hitting time of an arbitrary point $y$ in $S$ (Proposition \ref{P:exphit} below). Recall that the main use of this result is to derive concentration of local time (Lemma \ref{L:LTconc}) which in turn gives us estimates on the Radon-Nikodym derivative of $\Pr$ with respect to $\Q$, and hence on the partition function $Z(t, \beta)$.

Throughout we will use the notation $T_y : = \inf\{ t \ge 0 \ : \ X_t = y\}$ for the first hitting time of a vertex $y$.

\begin{prop}
  \label{P:exphit}
 Uniformly over all $x ,y\in S$, for some positive constants $c_{\ref{P:exphit}},C_{\ref{P:exphit}}$ depending only on the dimension $d$,
  \begin{equation}\label{exphit}
  \Q_x[T_y > t] \le
  \exp\left( -  c_{\ref{P:exphit}} {t \pi(y)}\right)
  \end{equation}
for all $t \ge C_{\ref{P:exphit}}/ \pi(y)$,  if $d \ge 3$. For $d = 2$, we get
\begin{equation}\label{exphit2_strong}
  \Q_x[T_y > t] \le
  \exp\left( -   c_{\ref{P:exphit}}  \frac{t}{\kappa \log( r+2) }  \right)
\end{equation}
 for all $t \ge C_{\ref{P:exphit}} \kappa \log( r+2)$, where $\kappa = \pi(y)^{-1}  \vee L^2 \log L$ and $r = \dist(y, \p S)$.
 \end{prop}

\begin{remark} \label{rem:exphit2_weak}
%Note that when $ d = 2$, $\pi(y) \le C L^{-2}$ so $\kappa \le O((\log L)  / \pi(y) )$, and thus it always holds
%\begin{equation}\label{exphit2_weak}
%\Q_x[T_y > t] \le
%\exp\left( -  c_{\ref{P:exphit}} \frac{t \pi(y)}{\log L  \cdot \log( r+2) }\right)
%\end{equation}
Note that when $d = 2$, it is always the case that $(1/\kappa) = \pi(y) \wedge 1/( L^2 \log L)$ satisfies
$$\frac{1}{\kappa} \ge c \frac{\pi(y)}{\log (r+2)}
$$
(consider the cases $\dist(y, \p S) \leq L/2$ and $\dist(y, \p S) \geq L/2$ to see this).
Hence
$c_{\ref{P:exphit} }$ can be chosen so that for all $y \in S$,
\begin{equation} \label{exphit2_weak}
\Q_x[T_y > t] \le
\exp \left( -  c_{\ref{P:exphit}} \frac{t \pi(y)}{ (\log( r+2) )^2 } \right)
\end{equation}
with $r=\dist(y,\p S)$.
\end{remark}

We start the proof of Proposition \ref{P:exphit} with a lemma which bounds the local time accumulated at a vertex $y$ until hitting another vertex $x$.
We introduce a box $B_1$ of side-length $L/100$ at macroscopic distance (of order $L$) away from $\partial S$; for now we will take $B_1 = [-L/200, L/200]^d$ but later we will allow $B_1$ to be centered at a different point such as $(\lfloor L/2\rfloor , \ldots, \lfloor L/2\rfloor)$.
We take $B_2$ a box of side length $L/50$ and $B_3$ a box of side-length $L/10$
both concentric to $B_1$.

\begin{lem}
\label{L:LThit}
Assume that $x \in B_1$ and $y \in S\setminus B_3$.
There is a constant $C_{\ref{L:LThit}} = C_{\ref{L:LThit}}(d)>0$ depending only on the dimension $d$ such that
$$
\E_y [ L( T_x , y) ] \leq
\begin{cases}
C_{\ref{L:LThit}} & \text{ if } d \ge 3\\
C_{\ref{L:LThit}} \log (r+2) & \text{ if } d = 2,
\end{cases}
 $$
 where $r = \dist(y, \p S)$
\end{lem}

\begin{proof}
Let $C_y\subset S$ be the cone formed by $y$ and $B_1$. Let $\Sigma_m = \set{ z \in S \cap C_y \ : \ \dist(z,y) = m }$, where $\text{dist}$ is the graph distance. Let $M = \inf\{m: \Sigma_m \cap B_2 \neq \emptyset \}$, note that $M =O(L)$ uniformly in $y \notin B_3$.
Let $\tilde C_y = \cup_{m=1}^M \Sigma_m$. Then let $\tilde C_x$ be the cone formed by $x$ and $\Sigma_M$ (see Figure \ref{F:cone}), and let $C = \tilde C_x \cup \tilde C_y$.
\begin{figure}\begin{center}
\includegraphics{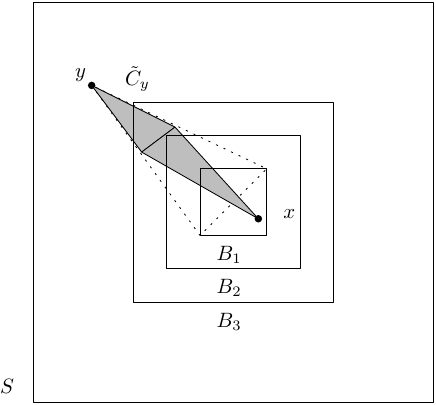}
\end{center}
\caption{The shaded area is $C$, the union of the two cones $\tilde C_y$  and $\tilde C_x$ in the proof of Lemma \ref{L:LThit}}
\label{F:cone}
\end{figure}

We will bound from below the probability that the Markov chain started from $y$ hits $x$ before returning to $y$.
Recall that our Markov chain with law $\Q$ is reversible with respect to $\pi$. Therefore it is equivalent to a discrete time random walk on a network
on $S$ where the weight, or conductance, of the edge $e = (z,w)$ is given by $\pi(z) Q(z,w) = \sqrt{ \pi(z) \pi(w)}$. It is equivalent in the sense that
both processes visit the same points in the same order, though possibly at different times.
Hence it suffices to bound from above the effective resistance $\Reff(y \to x)$ in this network.

We set all the weights on edges at distance greater than 1 from $C$ to be $0$, then
we have only reduced the conductance of all edges.
Rayleigh's monotonicity principle (see \cite[Chapter 2.4]{LyonsPeres})
tells us that the effective resistance only increases.
So it suffices to bound from above the effective resistance between $y$ and $x$
in this modified network.

The approach we use is that of \cite[Chapter 2]{LyonsPeres} (see, e.g., (2.17)).
Let $U$ be a random variable uniformly distributed on the base of the cone, $\Sigma_M$. Let $R$ be the union of the Euclidean segments $[y,U] $ and $[U,x]$.
Given $U$, let $\Gamma$ be some choice of a monotone path in $S$ that stays as close as possible from the two segments forming $R$, starting at $y$ and ending at $x$. By monotone we mean that each coordinate changes monotonically along $\Gamma$. ($\Gamma$ is thus a random monotone path connecting $x$ and $y$ through $\Sigma_M$, which stays at distance at most $\sqrt{d}$ from $R$; the exact way of choosing $\Gamma$ given $U$ does not matter).
Because $\Gamma$ is chosen to be monotone,
it traverses any edge at most once.
So the function $\theta(e) = \Pr [ e \in \Gamma ] - \Pr [ \hat{e} \in \Gamma ]$
(where $\hat{e}$ is the directed edge $e$
in the reverse direction) defines a unit flow in $C$, from $y$ to $x$. Indeed $\theta$ can be viewed as the expectation of a random variable which itself defines a unit flow almost surely. (Again, see
\cite[Chapters 2.4 \& 2.5]{LyonsPeres}, and especially (2.17)). Moreover, it is easily calculated that for an edge $e$ at distance $k$ from $y$ in $\tilde C_y$,
the probability that $e \in \Gamma$ is at most $O (k^{-(d-1)})$. Likewise, for an edge $e$ at distance $k$ from $x$ in $\tilde C_x$, the probability that $e$ is traversed by $\Gamma$ is at most $O(k^{-(d-1)})$.
Also, the number of edges in $\tilde C_y$ (resp.\ $\tilde C_x$) at distance $k$
from $y$ (resp.\ $x$) is at most $O( k^{d-1})$.

Thompson's principle (\cite[Chapter 2.4]{LyonsPeres})
then implies that the energy of this flow bounds the effective resistance. Noting that $\pi(z) \ge c \pi(y)$ for all $ z \in C$ we deduce
\begin{align*}
 \Reff(y \to x) &\leq \sum_{e} \mathrm{res}(e) \theta(e)^2  \\
 & \leq
\sum_{k =1}^M  \frac{1}{c \pi(y)} \cdot O(k^{-2 (d-1)} ) \cdot O(k^{d-1})  + \sum_{k =1}^{\diam(\tilde C_x)}  \frac{1}{c \pi(y)} \cdot O(k^{-2 (d-1)} ) \cdot O(k^{d-1}) \\
\\
& \leq
\begin{cases}
\frac{1}{c \pi(y)} \cdot O(\log L) & \text{ if } d=2\\
\frac1{c\pi(y)} & \text{ if } d \ge 3.
\end{cases}
\end{align*}
Thus, letting $w_y = \sum_{x \sim y } \sqrt{ \pi(x) \pi(y) }$,
\begin{align}
\label{eqn:escape from y}
\Q_y [ T_x < T_y ] & = \frac{1}{ w_y \Reff(y \to x) } \geq
\begin{cases}
 c ( \log L)^{-1}  & \text{ if } d=2\\
c & \text{ if } d \ge 3.
\end{cases}
\end{align}

In fact, when $ d = 2$ we can get a better bound by improving on the estimation of $\mathrm{res}(e)$ used above. Consider first the edges $e \in \tilde C_y$, and assume $y \in S_{r_0}$. It is obvious that for $r \ge r_0 +1$, $|S_r \cap \tilde C_y| = O(r-r_0)$.  Also, for each edge with at least one end in $S_r$, the probability that $e \in \Gamma$ is at most $O(r-r_0)^{-1}$.
Hence, if $r_0 \le L/2$,
\begin{align*}
\sum_{e \in \tilde C_y}  \mathrm{res}(e) \theta(e)^2 & \le  c \sum_{r=r_0+1 }^L \frac1{\mu_r}  \frac1{r-r_0} \\
& \le c L^3 \sum_{r=r_0+1}^{L/2} \frac{( \log (r+2) )^2}{r(r-r_0)} + c \sum_{r = L/2}^L \frac1{\mu_{L/2}( L/10)} \\
& \le c L^3 \frac{( \log (r_0+2))^3 }{r_0+2} + c L^2 ( \log L)^2 = \frac{c }{\mu_{r_0}} \log (r_0+2) + c L^2 ( \log L)^2 \\
& \le \frac{c }{\mu_{r_0}} \log (r_0+2) .
\end{align*}
In the second line above we have used that
\begin{align*}
\sum_{r=r_0+1}^{\infty} & \frac{ (\log (r+2) )^2 }{ r (r-r_0) }
%\leq \sum_{r =r_0+1}^{2 r_0} \frac{ (\log (r+2) )^2 }{ r (r-r_0) }  + \sum_{r > 2 r_0 } \frac{ (\log (r+2) )^2 }{ r (r-r_0) }
%\\
%&
\leq \frac{ 4 (\log (r_0+2) )^2}{ r_0+1 } \cdot \sum_{r=r_0+1}^{2 (r_0+1)} \tfrac{1}{r-r_0} +
\sum_{r>2(r_0+1)} \frac{2 (\log (r+2) )^2 }{ r^2 } .
\end{align*}
It is immediate that this conclusion also holds if $r_0 \ge L/2$. As for the edges in $\tilde C_x$, note that
\begin{align*}
\sum_{e \in \tilde C_x}  \mathrm{res}(e) \theta(e)^2 & \le  c \sum_{k=1 }^{L/100} \frac1{\mu_{99L/100}}  \frac1k
\le c L^2 \log L
\end{align*}
and note that this, too, is less or equal to $ c / \mu_{r_0}$. Therefore, we deduce that
$$
\Reff(y\to x) \le (C/ \mu_{r_0}) \log (r_0+2).
$$
Consequently,
\begin{align}
\label{eqn:escape from y2}
\Q_y [ T_x < T_y ] & = \frac{1}{ w_y \Reff(y \to x) } \geq  \frac{ c }{\log (r_0+2)}
\end{align}
for all $d\ge 2$.

The result follows easily by the strong Markov property and the fact that at each subsequent visit to $y$, the accumulated local time is an exponential random variable with rate bounded away from $0$ and thus has bounded mean.
\end{proof}

Let $ \tau_{xy} =  \inf\{ t \ge 0: \E_x[L(t,y) ] \ge 1\}$, and let $\tau = \max_{x \in B_1} \tau_{xy}$.  We immediately deduce from the above:

\begin{lem}\label{L:tau} Uniformly over $x \in B_1$ and $y \notin B_3$, we have
$$
\Q_x[T_y < \tau] \ge
 \begin{cases}
\frac{c}{ \log (r+2) } & \text{ if } d =2 \\
c & \text{ if } d\ge 3,
\end{cases}
$$
where $r = \dist(y, \p S).$
\end{lem}

\begin{proof}
It suffices to prove this with $\tau$ replaced by $\tau_{xy}$ since $\tau \ge \tau_{xy}$. Now, observe that
$$
\Q_x[T_y < \tau_{xy}] = \Q_x[L(\tau_{xy}, y)>0] = \frac{\E_x[ L(\tau_{xy}, y)]}{ \E_x[ L(\tau_{xy}, y) \ | \ L(\tau_{xy} , y) >0]}.
$$
Now, by definition,
$
\E_x[ L(\tau_{xy}, y)] = 1.
$
On the other hand, by the strong Markov property,
\begin{align*}
\E_x[ L(\tau_{xy}, y) \ | \ L(\tau_{xy} , y) >0]
& \le   \E_y[ L ( \tau_{xy}, y)]
\le \E_y[ L( T_x, y)] + \E_x[ L(\tau_{xy}, y )]  \\
& \le
\begin{cases}
C \log (r+2) + 1 & \text{ if } d =2 \\
C+1, & \text{ if } d\ge 3.
\end{cases}
\end{align*}
where $C= C_{\ref{L:LThit}}$ is the constant from Lemma \ref{L:LThit}. Thus, $\Q_x[T_y < \tau] \ge 1/ (C+1)$ if $d\ge 3$,
and $\Q_x[T_y < \tau] \ge 1/ ( C \log (r+2)+1)$ for $d =2$, as desired.
\end{proof}

We are now able to complete the proof of Proposition \ref{P:exphit}.

\begin{proof}[Proof of Proposition \ref{P:exphit}] Suppose first that $y \notin B_3$, and let $x$ be arbitrary in $S$.
By \eqref{tmix}, we know that $\tmix \le C_{\ref{tmix}}L^2 \log L$. It follows that, uniformly over $x \in S$, if $t = C_{\ref{tmix}} L^2 \log L$,
$$
\Q_x[ X_t \in B_1] \ge c.
$$
Define a sequence of times $t_1, t_2, \ldots$ by setting $t_n = n(C_{\ref{tmix}} L^2 \log L + \tau).$ Then uniformly over $x \in S$ and $y\notin B_3$, we obtain by Lemma \ref{L:tau} and the Markov property at time $t$,
$$
\Q_x [ T_y > t_1 ] \le 1-h,
$$
where $h = c$ if $d\ge 3$ and $h = c/\log (r+2)$ if $d = 2$. Hence, since this estimate is uniform in $x \in S$, we deduce by applying the Markov property at times $t_1, \ldots, t_n$,
$$
\Q_x[ T_y > t_n] \le (1-h)^n\le \exp ( - nh).
$$
Observe now that for $d\ge3$, $\tau \le c/\pi(y)$ for some $c$ large enough.
Indeed, $\pi(y) \le C/ L^{d}$ so if $t = c/\pi(y) $ with $c$ sufficiently large, then $t\ge 2 \tmix$ (see \eqref{tmix}). Hence we have that
$$
\E_x[L(t, y)] \ge \int_{t/2}^t \Q_x[X_s = y] ds \ge \tfrac{t}{2} \cdot \tfrac{\pi(y)}{2}  \ge 1,
$$
and hence it follows that for all $x \in B_1$ (and indeed all $x \in S$), $\tau_{xy} \le t$. Taking the maximum over $ x\in B_1$, we obtain as desired $\tau \le c/\pi(y)$. Observe further that, still in the case $ d\ge 3$, we have that $L^2 \log L \le C  /\pi(y)$ hence $t_1\le C /\pi(y)$ as well.

On the other hand, if $d = 2$, then the same argument gives $\tau \le C (L^2 \log L +1/\pi(y))$ for some $C>0$ large enough, so that we have $t_1 \le C \kappa$, where $\kappa = 1/ \pi(y) + L^2 \log L$.
Hence for $n = \lfloor c \pi(y) t\rfloor$ ($d\ge 3$) or $n = \lfloor c t/\kappa \rfloor$ ($d=2$),
$$
\Q_x( T_y > t ) \le \Q_x( T_y > t_n ) \le
\begin{cases}
 \exp ( - h t \pi(y)) & \text{ if } d \ge 3\\
 \exp ( - h t / \kappa) & \text{ if } d =2,
 \end{cases}
$$
as soon as $t\ge C/\pi(y)$ (for $d\ge 3$) or $t \ge C/\kappa$ ($d=2$) so that $n \ge 1$.

This immediately implies the result of Proposition \ref{P:exphit} if $y\notin B_3$. But the restriction $y\notin B_3$ is not essential.
Indeed if $y \in B_3$, we can always consider a disjoint cube $\tilde B_3$, also of side length $L/100$, and at macroscopic distance (of order $L$) away from $\partial S$, for instance the one centered at $(\lfloor L/2 \rfloor, \ldots, \lfloor L/2 \rfloor)$.
Throughout this box it will also be the case that $\pi(x) \ge c / L^{d}$ and so the exact same calculations apply,
yielding a similar conclusion for all $y\notin \tilde B_3$.
Since a given $y$ is either in $S\setminus B_3$ or in $S\setminus \tilde B_3$
(as $B_3$ and $\tilde B_3$ are disjoint), Proposition \ref{P:exphit} follows.
\end{proof}

\subsection{Tail estimates for local time}

We now turn to Lemma \ref{L:LTconc}, which proves concentration of the local time at an arbitrary point $y \in S$, for which the key input is the exponential tails derived in Proposition \ref{P:exphit}. We will then state a corollary summarising our main use of Lemma \ref{L:LTconc}.

\begin{lem} \label{L:LTconc}
There exist constants $C_{\ref{L:LTconc}},c_{\ref{L:LTconc}} >0$ depending only on $d$, such that
the following holds.
\begin{enumerate}
\item Assume $d\ge 3$. Uniformly in $x,y \in S$, for any $\delta > 2 / ( \pi(y) q_y t)$,
\begin{equation} \label{upperconc}
%\Q_x \Big( \left|\frac{L(t,y)}{\pi(y) t } -1 \right| \geq 2\delta \Big)
\Q_x \big( L(t,y) \geq (1+\delta) \pi(y) t \big)
 \leq
C_{\ref{L:LTconc}} \exp \big(- c_{\ref{L:LTconc}} ( \sqrt{\delta} \wedge \delta^2) \pi(y) t \big) .
\end{equation}

\item Assume $d=2$. Uniformly in $x,y \in S$, for any $\delta > 2/( \pi(y) q_y t)$,
\begin{equation}
%\Q_x \Big( \left|\frac{L(t,y)}{\pi(y) t } -1\right|  \geq 2\delta \Big)
\Q_x \big( L(t,y) \geq (1+\delta) \pi(y) t \big) \le C_{\ref{L:LTconc}} \exp \big(
- c_{\ref{L:LTconc}} (\sqrt{\delta}\wedge \delta^2) \frac{\pi(y) t}{ ( \log (r+2))^2} \big)
\end{equation}
where $r = \dist(y, \p S)$.
\end{enumerate}
\end{lem}

\begin{proof}
Fix $y \in S$. In this proof it is convenient to define a time $\theta$ by putting
\begin{equation}\label{Delta}
  \theta = \begin{cases}
  %c_{\ref{P:exphit}}\kappa \log (r+2)   & \text{ if } d =2 \text{ where } \kappa = \frac1{\pi(y)} \vee  L^2 \log L\\
\frac{ c_{\ref{P:exphit}} }{\pi(y) } (\log (r+2) )^2
& \text{ if } d =2  \\
\frac{c_{\ref{P:exphit}}}{\pi(y)} & \text{ if } d \ge 3.
\end{cases}
\end{equation}
%\begin{equation}\label{Delta}
%  \theta = \begin{cases}
%   K \kappa \log r   & \text{ if } d =2 \text{ where } \kappa = \frac1{\pi(y)} \vee  L^2 \log L\\
%   K \frac{1}{\pi(y)} & \text{ if } d \ge 3.
%    \end{cases}
%\end{equation}
By Proposition \ref{P:exphit} (and Remark \ref{rem:exphit2_weak}),
$y$ is hit with positive probability every $\theta$ units of time.
%The constant $K$ will be chosen large enough depending only on the constant $c_{\ref{P:exphit}}$
%from Proposition \ref{P:exphit}.

Let $q_y = \sum_{x} Q(y,x)$ be the total jump rate from $y$ under $\Q$.
Note that $q_y$ is of constant order for $L$ sufficiently large.
Fix $\eps>0$ (in a way which will depend on $\delta$ and will be specified below), and let $n = \lceil \pi(y) q_y t  (1+\eps) \rceil$.

It will be useful to define
$$ T = \inf \set{ t \geq 0 \ : \ X_t =y } \qquad
T^+ = \inf \set{ t \geq T_{S \setminus \set{y} }  \ : \ X_t = y } - T_{S \setminus \{y\} } , $$
which are the hitting and return time to $y$,
and also the successive return times to $y$:
$T_0=T$, and for $k>0$,
$$ \tilde T_k = \inf \set{ t \geq T_{k-1} \ : \ X_t \neq y }
\qquad
T_k = \inf \set{ t \geq \tilde T_k \ : \ X_t = y } . $$
%Clearly $L_{\tilde X}(t,y) \leq n$ if and only if $T_n \geq t$.

Note that $T_n$ is the sum of the independent increments
\begin{align} \label{eqn:Tn}
T_n & = T_0 + \sum_{j=1}^n T_j - T_{j-1}
= T_0 + \sum_{j=1}^n T_j - \tilde T_j + \sum_{j=1}^n \tilde T_j - T_{j-1}
.
\end{align}
For each $j$, the increments $T_j - \tilde T_j$ have the same law.
Also, the second sum
$$ \sum_{j=1}^n \tilde T_j - T_{j-1} $$
is just $L(\tilde T_n , y) = L(T_n,y)$, each increment having the law of an independent exponential
random variable of rate $q_y$.

{\bf Step I.} First, we bound $\Q_x (T_n \leq t)$ by bounding
the first sum $\sum_{j=1}^n T_j - \tilde T_j$.
Note that $T_j - \tilde T_j$ has the distribution of $T^+$ under $\Q_y$.
Thus,
by Proposition \ref{P:exphit} and Remark \ref{rem:exphit2_weak} (for the $d=2$ case),
$T_j - \tilde T_j$ has an exponential tail
$\Q_x ( T_j - \tilde T_j > t ) = \Q_y (T^+ > t ) \leq e^{- \theta^{-1} t}$
for $t \ge C_{\ref{P:exphit} }\theta$,
and
\begin{align*} \E_x [ (T_j - \tilde T_j)^2 ] &=\E_y [ (T^+)^2] =  \int_0^\infty 2t\Pr_y[T^+ >t] dt\\
&\leq
\int_0^{C_{\ref{P:exphit} }\theta} 2t dt + \int_{C_{\ref{P:exphit} }\theta}^\infty 2t e^{ - t /\theta} dt
\\
& \leq A : =
( (C_{\ref{P:exphit} } )^2 + 2 ) \cdot \theta^2 .
\end{align*}
Also, it is well known that $\E_y[T^+] = \tfrac{1}{q_y \pi(y) }$, so
$$ \sum_{j=1}^n \E_x [T_j - \tilde T_j] = n \E_y[T^+] \geq t(1+\eps) $$
by our choice of $n$.
Using the inequalities $e^{-\xi} \leq 1 - \xi + \xi^2$, valid for for $\xi>0$,
and $1 + \xi \leq e^\xi$, valid for any $\xi \in \R$,
we deduce that for any $\alpha>0$,
\begin{align}
\E_x [ e^{-\alpha T_n } ] & \leq \big( \E_y [ e^{-\alpha T^+ } ] \big)^n
\leq \big( 1 - \alpha \E_y [T^+ ] + \alpha^2 A  \big)^n
\leq \exp \big( - \alpha n  \E_y[T^+] +  n \alpha^2 A  \big) .
\label{eq:upperconc}
\end{align}
Since
$\Q_x ( T_n \leq t )
\leq \Q_x \big( e^{-\alpha T_n} \geq e^{-\alpha t} \big)$,
we have
\begin{align*}
\Q_x ( T_n \leq t )
& \leq \exp \big( \alpha^2 n A + \alpha (t - n \E_y[T^+] ) \big) ,
\end{align*}
which we may optimise over $\alpha >0$.
We find that the right hand side is minimised for $\alpha = \frac{ n \E_y[T^+] -t}{2 n A}$ (note that $\alpha \ge t \eps/ (2nA) >0 $).
Substituting, this implies
\begin{align}
\label{eqn:first step}
\Q_x ( T_n \leq t ) & \leq \exp \Big( - \frac{ (n \E_y[T^+] - t )^2  }{ 4 n A } \Big) \leq \exp \Big( - \frac{ \eps^2 t^2   }{ 4 n A } \Big)
 \nonumber \\
&\leq \exp \Big(  -  \frac{\eps^2  }{ 4 (2 + (C_{\ref{P:exphit} } )^2 ) }
\cdot \frac{t^2}{ n  \theta^2 }   \Big) .
\end{align}

{\bf Step II.}  Next, we bound $\Q_x (T_n > t , L(t,y) \geq \pi(y) t (1+\delta) )$
by bounding the second sum $\sum_{j=1}^n \tilde T_j - T_{j-1}$ in \eqref{eqn:Tn}.

%Note that $L(t,y) \geq \pi(y) t (1+\delta)$ implies that
%either $T_n \leq t$ ({\em i.e.}\ $y$ has been visited at least $n$ times),
%or $L(T_n, y) \geq \pi(y) t \lambda$.
Note that
$$ \sum_{j=1}^{n} (\tilde T_j - T_{j-1}) = L(\tilde T_n , y) = L(T_n,y) $$
has the distribution of $\sum_{j=1}^n E_j$ where
$(E_j)_j$ are i.i.d.\ exponential random variables of rate $q_y$.
Standard concentration bounds on sums of i.i.d.\ exponential random variables show
that for any $\eta>0$,
\begin{align*}
%\label{eqn:sum of exp's}
\Pr \big(\sum_{k=1}^n E_k \geq  n \tfrac{1+\eta}{q_y} \big) & \leq \exp \big( - n ( \eta- \log(1+\eta) )  \big) .
\end{align*}
Thus, if $\eps>0$ satisfies
\begin{equation}\label{constraint}
\pi(y) q_y t (1+\delta) \geq n (1+\eps)
\end{equation}
then
\begin{align}
\label{eqn:second step}
\Q_x ( L(T_n,y) \geq \pi(y) t (1+\delta) ) & \leq
\Q_x ( L(T_n,y) \geq n \tfrac{1+\eps}{q_y} ) \leq
\exp \big( - n ( \eps -  \log(1+\eps) )  \big) .
\end{align}
Finally,
we have that the event $\{ L(t,y) \geq \pi(y) t (1+\delta) \}$
implies that either $T_n \leq t$ or $L(T_n,y) \geq L(t,y) \geq \pi(y) t (1+\delta) \geq n (1+\eps)/q_y$,
still assuming that $\eps$ satisfies the constraint \eqref{constraint}.
%Note that %when $L$ is large enough $q_y \leq 1$ and
%$\pi(y) t \geq C \beta$, so $\eps$ may be chosen by requiring
%$\tfrac{(1+\eps)}{C \beta} + (1+\eps)^2 \leq 1+\delta$.
%Thus, when $\delta > \tfrac{2}{C \beta}$ we may require $4\eps + 2\eps^2 \leq \delta$.
%This tells us that $\eps \asymp \delta \wedge \sqrt{\delta}$.
Since $n \le \pi(y) q_y t(1+ \eps) + 1$ and $\delta  \ge 2 / (\pi(y) q_y t)$ by assumption on $\delta$ in the theorem, this is the case as soon as
$$ (2\eps + \eps^2) \leq \frac{\delta /2}{1 + 1/(\pi(y) q_y t)} $$
%By adapting the constant $C_{\ref{P:exphit}}$, and since $q_y \to 1$ uniformly over $S$ as $L$ becomes large, we may assume without loss of generality that
%$\tfrac{1}{ \pi(y) q_y t} < \delta/2$.
Note that ${\pi(y) q_y t} \ge  \beta_0 C_{\eqref{eqn:2D trans prob}} /2$ as $q_y \to 1$ uniformly and $t \ge \beta L^{d+1}$.
Hence we can choose
$$ \eps \asymp \min ( \delta, \sqrt{\delta}) $$
so that \eqref{constraint} is satisfied, where the implied constants depend on $\beta_0$ and $d$ only.
Combining \eqref{eqn:first step} and \eqref{eqn:second step} we arrive at the conclusion of the lemma,
since
\begin{align}
 \Q_x \big( L(t,y) \geq \pi(y) t (1+\delta) \big) & \leq \Q_x \big( T_n \leq t \big)
+ \Q_x \big( L(T_n,y) \geq \pi(y) t (1+\delta) \big) \nonumber \\
& \leq \exp \Big( - c \tfrac{ \eps^2}{1+\eps} (\log (r+2) )^{-2}   \pi(y) t \Big)  + \exp \Big( - c  \eps^2 \pi(y) t \Big)\label{eq:largedev}\\
& \leq 2 \exp \Big( - c ( \eps^2\wedge \eps)  \cdot (\log (r+2) )^{-2} \cdot  \pi(y) t \Big) ,\nonumber
\end{align}
where the $(\log (r+2) )^{-2}$ term can be removed when $d \geq 3$ and $c>0$
is a constant depending on $ C_{\eqref{eqn:2D trans prob}}$, $C_{\ref{P:exphit} } , c_{\ref{P:exphit} }$, $\beta_0$ and $d$.
(We have used \eqref{eqn:second step} to get \eqref{eq:largedev} and the fact that
$(\eps - \log(1+\eps) )\asymp \eps^2 \wedge \eps $ as well as $n \le (1+ \eps) \pi(y) t$.)
The lemma now follows since $\eps \asymp \min ( \delta, \sqrt{\delta})$, so that $\eps^2 \wedge \eps \asymp \min ( \sqrt{\delta}, \delta, \delta^2) \asymp \min (\sqrt{\delta}, \delta^2)$.
\end{proof}

\begin{remark}\label{R:LTconc_low}
  A similar statement to Lemma \ref{L:LTconc} holds with the upper bound replaced by a lower bound:
  %namely, there exist constants $c_{\ref{R:LTconc_low}}, C_{\ref{R:LTconc_low}}$ depending only on the dimension and $\beta_0$, such that for $d \ge 3$, for any $\delta \in (0,1)$ with $\delta > 2/(\pi(y) q_y t)$:
  %\begin{equation} \label{lowerconc}
%%\Q_x \Big( \left|\frac{L(t,y)}{\pi(y) t } -1 \right| \geq 2\delta \Big)
%\Q_x \big( L(t,y) \leq (1-\delta) \pi(y) t \big)
% \leq
%C_{\ref{R:LTconc_low}} \exp \big(- c_{\ref{R:LTconc_low}} \delta^2 \pi(y) t \big),
%\end{equation}
%and for $d=2$, uniformly in $x,y \in S$,
%\begin{equation}
%%\Q_x \Big( \left|\frac{L(t,y)}{\pi(y) t } -1\right|  \geq 2\delta \Big)
%\Q_x \big( L(t,y) \leq (1-\delta) \pi(y) t \big) \le C_{\ref{R:LTconc_low}} \exp \big(
%- c_{\ref{R:LTconc_low}} \delta^2 \frac{\pi(y) t}{ ( \log (r+2))^2} \big).
%\end{equation}
The proof is essentially similar with a few additional complications because we can no longer use the simple bound $e^{-\xi}\le 1- \xi + \xi^2$ which was valid for all $\xi\ge 0$, but when $\xi \le 0$ is only valid for $-1 \le \xi \le 0$ (see \eqref{eq:upperconc}). However, in order to not overload the paper with technical details, and since this isn't needed for the proof of Theorem \ref{T:2dimGibbs} we have chosen not to include the proof.
\end{remark}

\begin{cor}
\label{cor:tail}
For all $\beta_0>0$ there exists a constant $c_{\ref{cor:tail}} > 0$ depending only on $d$ and $\beta_0$, such that
the following holds.  For any integer $k > 0$ and any $x,y \in S$,
\begin{equation*}
\Q_x \Big( L(t,y) \geq 2^k \pi(y) t \Big)
 \leq
2 \exp \Big( - c_{\ref{cor:tail}} 2^{k/2} \cdot \beta_0 r ( \log (r+2) )^{-4}   \Big) ,
\end{equation*}
where
$r = \dist(y , \p S)$.
\end{cor}

\begin{proof}
This just follows from taking $1+\delta = 2^k$ in
Lemma \ref{L:LTconc},
%in which case we get $\eps \asymp 2^{k/2}$,
%so $\tfrac{\eps^2}{1+\eps} \asymp 2^{k/2}$.
where we also use the facts that $q_y \to 1$ uniformly over $S$, and that $t= \beta L^{d+1}$ so that $\pi(y) t \geq C_{ \eqref{eqn:2D trans prob}} \beta_0 r ( \log(r+2) )^{-2}$.
\end{proof}

\section{Proof of Theorem \ref{T:2dimGibbs}}

The goal of this section is to obtain the following lower bound on the partition function.
\begin{prop}\label{P:lbZ} Let $\beta_0>0$ be fixed and let $\beta > \beta_0$. Then
$$
Z(t,\beta) \ge
\exp\left(- \gamma t^{1- 2/(d+1)} \beta^{2 / (d+1)} \right)
$$
 where $\gamma$ is a constant depending only on $\beta_0$ and $d\ge 2$.
\end{prop}

\subsection{Good event}

For any $0 < r \leq L$ recall the definition of $S_r$:
$$ S_r = \{z \in S \ : \ \dist(z, \p S) = r\} = \{ z \in S \ : \ \| z\|_\infty = L-r \} . $$
For $z \in S_r$ let $\IP{z} = \# \{ 1 \leq j \leq d \ : \ |z_j| = \| z \|_\infty \}$
(which is between $1$ and $d$).
Define
\begin{equation}
\label{Drdef} D_r = \{ z \in S_r \ : \ \IP{z} > 1 \} .
 \end{equation}
In two dimensions the vertices of $D_r$ are exactly the four corners of the square $S_r$, while in three dimensions these are the edges of the cube defined by $S_r$. More generally the vertices of $D_r$ are those which are in the intersections of faces of the hypercube defined by $S_r$.

For any $k \ge 1$ define the (random) subset
$$\cX_k = \{ x \in S \ : \  \frac{L(t,x) }{ t \pi(x)} \in [2^k, 2^{k+1})\} , $$
and consider the set of vertices
$$
S_{r,k} = S_r \cap \cX_k \qquad D_{r,k} = D_r \cap \cX_k.
$$
Define the ``good'' events:
$$
\cS_{r,k} = \{ |S_{r,k}| \le |S_r| \exp( - c_{\ref{L:A}} 2^{k/2} r (\log (r+2))^{-4}) \}
$$
where $c_{\ref{L:A}} = \tfrac12 c_{\ref{cor:tail}}$. Define:
$$
\cS_r = \bigcap_{k \ge k_{\ref{L:A}}} \cS_{r,k} \qquad  \cS = \bigcap_{r=1}^L \cS_r
$$
where $k_{\ref{L:A}}$ will be chosen below, large enough. Likewise, define
$$
\cD_{r,k} = \{ |D_{r,k}| \le |D_r| \exp( - c_{\ref{L:A}} 2^{k/2} r (\log (r+2))^{-4})\}
$$
where $c_{\ref{L:A}} =  \tfrac12 c_{\ref{cor:tail}}$,
and
$$
\cD_r = \bigcap_{k \ge k_{\ref{L:A}}} \cD_{r,k} \qquad  \cD = \bigcap_{r=1}^L \cD_r
$$
as above.

Fixing some $c_{\eqref{cBoundary}}$ large enough (which will be chosen later)
we define the event
%\begin{align}
%\cA_r&:= \{ L(t,S_r) \leq  \gamma \pi(S_r) t \} \ \  ; \ \  \cA = \{ R_t %\subset S\} \cap \bigcap_{r=1}^L \cA_r.
%\end{align}
%We also define the event
\begin{align}\label{cBoundary}
\cB &:= \{|\p R_t| \le
c_{\eqref{cBoundary}} L^{d-1}
 \} ,
\end{align}
% and the events
% \begin{equation}
%\cC_r =
%\begin{cases}
%\bigcap_{x \in S_r} \set{ L(t,x) \leq (\log L)^\gamma \pi(x) t } & \textrm{ %for } r < (\log L)^4 , \\
%\bigcap_{x \in S_r} \set{ L(t,x) \leq \gamma \pi(x) t } & \textrm{ for } r %\geq (\log L)^4 ,
%\end{cases}
%\end{equation}
%and
%$\cC = \bigcap_{r=1}^L \cC_r$.
and finally, we define the good event $\cG$ as follows:
\begin{equation}\label{G}
  \cG =
\cB \cap \cS  \cap \cD.
\end{equation}

We will now proceed to show that the probability of the good event $\Q_x(\cG)$ is bounded below uniformly in $L$. %We start with the following lemma.
We will allow the starting point to be any fixed arbitrary $x\in S$ (although we only require these results with $x = 0$).

\begin{lem}\label{L:A}
Fix $\eps>0$ and $\beta_0>0$. We can choose $k_{\ref{L:A}} = k_{\ref{L:A}}(\beta_0,\eps)$ such that for
any $\beta \geq \beta_0$,  %depending only on $\beta_0$
and for all $t$ sufficiently large,
we have $\Q_x(\cS) \ge 1-\eps $ and $\Q_x(\cD) \ge 1-\eps$ for all $x \in S$.
\end{lem}

\begin{proof}
  We only show the proof for $\cS$, as the proof for $\cD$ is very similar.  By Corollary \ref{cor:tail}, taking expectation under $\Q_x$,
  \begin{align*}
  \E_x |S_{k,r} |  & \le |S_r| \max_{y \in S_r} \Q_x (y \in \cX_k)
  %\\ &
  \le |S_r| \max_{y \in S_r}\Q_x (L(t,y) \ge 2^k \pi(y) t ) \\
  & \le |S_r|  \exp \big( -  2c_{\ref{L:A}} 2^{k/2} r (\log (r+2))^{-4} \big).
  \end{align*}
Applying a union bound and Markov's inequality, we deduce that
 $$
 \Q_x (\cS^c) \le \sum_{k \ge k_{\ref{L:A}}} \sum_{r \ge 1}
 \exp \big( -  c_{\ref{L:A}}  2^{k/2} r  (\log (r+2))^{-4} \big)
 $$
 and so can be made arbitrarily small by choosing $k_{\ref{L:A}}$ large enough
 (depending only on $\beta_0$), as desired.
\end{proof}

Now, we estimate $\E_x  |\p R_t|$ under $\Q_x$.

\begin{lem}\label{L:ER}
Let $\beta_0 > 0$.  There exists
$C_{\ref{L:ER}}> 0$ (depending only on $\beta_0$)
such that for all $\beta \geq \beta_0$,
under $\Q_x$ we have
$\E_x |\p R_t |  \le
C_{\ref{L:ER}} L^{d-1}$.
\end{lem}

\begin{proof}
Since the maximal degree is $2d$ we obtain that $|\p R_t| \leq 2d |S \setminus R_t| + CL^{d-1}$, where the second term represents all vertices on $\p S$.
Note that if $y \in S_r$ then $\pi(y) t \geq c_{\ref{P:exphit}} \beta r (\log (r+2) )^{-2}$ for some constant $c>0$.
Using Proposition \ref{P:exphit} (and Remark \ref{rem:exphit2_weak}),
\begin{align*}
\E_x  | S \setminus R_t |  & = \sum_{y \in S} \Q_x [ y \not\in R_t ] = \sum_{y \in S} \Q_x [ T_y > t ]
\leq \sum_r |S_r|
\exp \sr{ - c_{\ref{P:exphit}} \beta_0 r (\log (r+2) )^{-4} } \leq C L^{d-1} ,
\end{align*}
as desired.
\end{proof}

We deduce from Lemma \ref{L:ER} and Markov's inequality that
$$
\Q_x\left(|  \p R_t | \ge
2C_{\ref{L:ER}} L^{d-1} \right) \le 1/2.
$$
In particular, together with Lemma \ref{L:A}, if we take $c_{\eqref{cBoundary}} \ge 2 C_{\ref{L:ER}}$
%in addition to previous requirements
(so altogether $c_{\eqref{cBoundary}}$ is chosen large enough in a way which depends only on $\beta_0$ and $d$), we obtain for $L$ sufficiently large
\begin{equation}
\label{eq:G} \Q_x( \cG) \ge 1/4.
\end{equation}

\subsection{Radon-Nikodym derivative estimates}
The following lemma is well known but very useful, see e.g. \cite{RogersWilliams}, IV, (22.8). We include it for completeness.
\begin{lem}\label{L:PQ}
Let $f(z) = \sqrt{\pi(z)}$. Let $\Delta f (x) = \sum_{y \sim x} f(y) - f(x)$ be the discrete Laplacian.
Then
$$
 \frac{d\Pr_x}{d\Q_x}|_{\cF_t} = \frac{f(X_0)}{f(X_t)} \exp\left( \int_0^t \frac{\Delta f}{f}(X_s) ds\right).
$$
\end{lem}
\begin{proof}
This follows easily from a discrete Feynman--Kac representation (see e.g. Lemma 11 in \cite{GH}). An alternative elementary proof is as follows.
Suppose the successive states visited by $\omega$ up to time $t$ are $x_0, \ldots, x_n$, with the path staying a time $\tau_0, \ldots, \tau_n$ at respectively at these locations. (Hence $\tau_1 + \ldots + \tau_n = t$.)  If $x \in \Z^d$, then the total rate at which the particle would jump out of $x$ under $\Q$ is given by $q(x) = f(x)^{-1} \sum_{y \sim x} f(y)$. Then  letting $d(x) = 2d$ be the total rate of leaving $x$ under $\Pr$,
  \begin{align*}
  \frac{d\Pr_x}{d\Q_x}( \omega) & = \frac{e^{- d(x_0) \tau_0} \ldots e^{- d(x_n) \tau_n}}{Q(x_0, x_1) e^{- q(x_0) \tau_0} \ldots Q(x_{n-1}, x_n) e^{- q(x_n) \tau_n}}\\
  & = \frac{f(x_0)}{f(x_n)} \prod_{i=0}^n \exp \big( (q -d)(x_i) \tau_i \big)
  = \frac{f(x_0)}{f(x_n)} \prod_{x \in \Z^d} \exp \big( (q-d)(x) L(t,x) \big) \\
  & = \frac{f(x_0)}{f(x_n)} \exp \left( \sum_{x \in \Z^d} \frac{\Delta f(x)}{f(x)} L(t,x)\right).
  \end{align*}
The result follows immediately.
\end{proof}

\begin{lem} \label{lem:RN bound}
Recall the events $\cS, \cD$ defined above \eqref{G}.
On the event $\cS \cap \cD$ we have
$$ \int_0^t \frac{ \Delta f}{f} (X_s) ds \geq - c_{\ref{lem:RN bound}} t L^{-2} , $$
where $c_{\ref{lem:RN bound}}  > 0$ is some constant (depending only on the dimension $d$ and on $\beta_0$).
\end{lem}

\begin{proof}
To ease the presentation, write $\mu(r) = \mu_r$, and consider $\mu$ as a function on real positive numbers. We want to estimate $\sum_x \tfrac{\Delta f(x)}{f(x)} L(t,x)$ from below. The terms $x \in S_r \setminus D_r$, $r \neq L/2$ are the ``main terms" and all the other terms ($ r= L/2$ or $x \in D_r$) are a kind of error which we need to estimate.

\medskip \noindent \textbf{Step 1: contribution of main terms.}
We will show that
\begin{equation}\label{Srcontr}
\sum_{r \neq L/2} \sum_{x \in S_r \setminus D_r}  \tfrac{\Delta f(x)}{f(x)} L(t,x) \ge - c \cdot \tfrac{t}{L^2}
\end{equation}

Note that $\Delta f(x) / f(x)$ does not change if we mutiply $f$ by a nonzero constant. Hence for this calculation we may take $C_{\eqref{eqn:2D trans prob}} = 1$ in the definition of $\pi(x)$. Thus we have for $x \in S_r$
$$ f(x) = \sqrt{\mu(r)} =
\begin{cases}
L^{-(d+1)/2} \frac{\sqrt{r} }{\log (r+2)} & \textrm{ if } r \leq L/2 \\
\sqrt{ \mu(L/2) } +  \tfrac{r-L/2}{ L^{(d+2)/2} } & \textrm{ if } r > L/2 .\\
\end{cases}
$$
%where $C$ is such that $\sum_x f(x)^2 = 1$.
%
%Let $g(x) = \log f(x)$.
%Note that for $r \leq L/2$,
%$$ (\log \mu(r))' = \tfrac1r - \tfrac{2}{(r+2) \log(r+2)}
%$$
%and
%$$ (\log \mu(r))'' = - r^{-2} + \tfrac{2 \log(r+2) + 2}{ (r+2)^2 ( \log (r+2))^2 }
%= - r^{-2} \cdot (1 - O( (\log(r+2) )^{-1} )  ) . $$
%
%Thus, if $x \in S_r, y \in S_{r+1}$, then for some $\xi \in [r,r+1]$,
%\begin{align*}
%2(g(y)- g(x)) & = (\log \mu(r))' + \tfrac12 (\log \mu(\xi))''
%= \tfrac1r - \tfrac{2}{(r+2) \log(r+2)} - O(r^{-2}) .
%\end{align*}
%This implies that
A second order Taylor expansion provides the following estimate for all $x \in S_r, y \in S_{r+1}$
with $r+1 \leq L/2$:
\begin{align*}
\tfrac{f(y)}{f(x)} - 1 & = \tfrac1{2r} - \tfrac{1}{(r+2) \log(r+2)} + O(r^{-2}) .
\end{align*}
Now, for $x \in S_r \setminus D_r$ with $r+1 \leq L/2$,
all neighbours are in $S_r$ (and hence do no contribute to the Laplacian) except for one in $S_{r+1}$ and one in $S_{r-1}$.
Hence, for some $\xi \in [r-1,r]$,
\begin{align*}
\tfrac{\Delta f (x)}{f(x)} & = \tfrac1{2r} - \tfrac{1}{(r+2) \log(r+2)} - \tfrac1{2(r-1)} + \tfrac{1}{(r+1) \log(r+1)} + O(r^{-2}) \\
& = -\tfrac{1}{2 \xi^2} + \tfrac{\log(\xi+2) + 1}{ (\xi+2)^2 ( \log (\xi+2))^2 } + O(r^{-2} ) \geq - c r^{-2} ,
\end{align*}
for some constant $c>0$. For $r > L/2$, we have
that if $x \in S_r \setminus D_r$, then
$\Delta f(x)   = 0$ since $f$ is affine in this range.
%true that - c/r^2 but it would be hard to see with above techniques...!
%The terms $x \in S_r \setminus D_r$, $r \neq L/2$ are the ``main terms" and all the other terms ($ r= L/2$ or $x \in D_r$) is a kind of error which we need to estimate.
Hence let us estimate the contribution to the Radon--Nikodym derivative \eqref{Srcontr}
coming from points in $S_r\setminus D_r$.
Denote
$$ \vphi(r) =
\begin{cases}
r (\log (r+2) )^{-2}  & \textrm{ if } r \leq L/2 , \\
r & \textrm{ if } r > L/2 .
\end{cases}
$$
If $x \in S_{r,k} \setminus D_r$ we have that $\tfrac{ \Delta f(x) }{f(x) } \geq - c r^{-2}$
and also
$L(t,x) \leq 2^{k+1} \pi(x) t \leq c 2^{k+1} (t/L^{d+1}) \vphi(r)$.
So, on the event $\cS_{r,k}$,
\begin{align*}
\sum_{x \in S_{r,k} \setminus D_r } & \tfrac{ \Delta f(x) }{ f(x) } L(t,x) \geq
-c |S_{r,k} \setminus D_r| \cdot 2^{k+1} \tfrac{t}{L^{d+1} }\vphi(r) r^{-2}
\\ &
\geq -c |S_r| \cdot 2^{k+1}  \tfrac{t}{L^{d+1}}  \vphi(r) r^{-2}  \cdot
\exp \big( - c_{\ref{L:A}} 2^{k/2} r (\log (r+2) )^{-4}     \big) .
\end{align*}
Summing over $k$ and since $|S_r| \le L^{d-1}$, we obtain that on the event $\cS$,
\begin{align}
\sum_{x \in S_r \setminus D_r}  \tfrac{ \Delta f(x) }{ f(x) } L(t,x)
 &\geq -tL^{-2} \cdot  \vphi(r) r^{-2}  \cdot
\Big( %k_{\ref{L:A} }
2^{k_{\ref{L:A} } +1} +
\sum_{k \geq k_{\ref{L:A}} } 2^{k+1} e^{ - c_{\ref{L:A}} 2^{k/2} r (\log (r+2) )^{-4}   } \Big) \nonumber\\
&
\geq - c t L^{-2}  \vphi(r) r^{-2} ,\nonumber
\end{align}
where the final constant $c>0$ depends on $k_{\ref{L:A}}$. Hence, summing over $r \neq L/2$, the contribution of the main terms to the Radon--Nikodym derivative is
\begin{equation*}
\sum_{r \neq L/2}\sum_{x \in S_r  \setminus{ D_r}} \tfrac{\Delta f(x) }{ f(x) } L(t,x) \ge - c t L^{-2}   %\label{RN}
\end{equation*}
as desired in \eqref{Srcontr}, because
\begin{equation}
\label{justifprofile}
\sum_{r=1}^L \vphi(r) r^{-2} \leq
\sum_{1\leq r \leq L/2} \tfrac{1}{r (\log (r+2))^2 } + \sum_{L/2 < r \leq L} \tfrac1r \leq C .
\end{equation}

\medskip \noindent \textbf{Step 2: contribution of $\cD_r, r < L/2$.} %We start with estimating the error coming from terms where $x \in D_r$.
%Recall that $S_r$ can be described as $\{x \in S: \|x\|_\infty = L-r+1\}$.
%For a point $x \in S$, let $ \IP{x} = \# \set{ 1\leq j \leq d \ : \ |x_j| = ||x||_\infty }$
%be the number of coordinates which achieve the maximum.
%Then the ``diagonal'' $D$ is given by $D = \{ x \in S: \IP{x} >1\}$,
%the set of points for which the maximum is attained by more than one coordinate.
%Note that with this definition, if $x \in D_r := S_r \cap D$, there are exactly $\IP{x}$ edges leading from $x$ to $S_{r-1}$,
%and $2 d-\IP{x}$ edges that lead from $x$ to $S_r$.
If $x \in D_r$ for $r < L/2$, then $2d - \ip{x}$ neighbours of $x$ are in $S_r$, and $\ip{x}$ neighbours are in $S_{r-1}$ (here recall that $\ip{x} = \# \set{ 1\leq j \leq d \ : \ |x_j| = ||x||_\infty }$
is the number of coordinates which achieve the sup norm of $x$, as defined in \eqref{Drdef}).
So,
\begin{align*}
\tfrac{\Delta f(x) }{f(x)} & = -\tfrac{\ip{x} }{2(r-1)} + \tfrac{\ip{x} }{(r+1) \log(r+1)} + O(r^{-2}) \geq -d \cdot c (r+2)^{-1} .
\end{align*}
Hence noting that $|D_r| = O( L^{d-2})$, on the event $\cD$ the contribution to \eqref{Srcontr} coming from $ D_r, r< L/2,$ is:
\begin{align}
\sum_{x \in D_r}  \tfrac{ \Delta f(x) }{ f(x) } L(t,x) & \geq
-c |D_r| \cdot \tfrac{t}{L^{d+1}} \vphi(r) r^{-1} \cdot
\Big( %k_{\ref{L:A} }
2^{k_{\ref{L:A} } +1} +
\sum_{k \geq k_{\ref{L:A}} }
2^{k+1} e^{ - c_{\ref{L:A}}  2^{k/2} r (\log (r+2) )^{-4}   } \Big)
\nonumber \\
& \geq - c L^{d-2} \cdot \tfrac{t}{L^{d+1}} \vphi(r) r^{-1} \geq -c t L^{-3} \vphi(r) r^{-1}  \nonumber %\geq - c t L^{-2} \vphi(r) r^{-2} .
\end{align}
and thus summing over $r< L/2$ we get
\begin{equation}
\label{Drcontr<}
\sum_{r < L/2}\sum_{x \in D_r}  \tfrac{ \Delta f(x) }{ f(x) } L(t,x)  \ge - c \tfrac{t}{L^2 ( \log L)^2}
\end{equation}

\medskip \noindent \textbf{Step 3: contribution of $\cD_r$ with $r \ge L/2$.}
If $x \in D_r$ for $r \ge L/2$ then
\begin{align*}
\tfrac{\Delta f(x)}{f(x)} %& = \ip{x} \cdot \frac{ \sqrt{\mu(r-1)} - \sqrt{\mu(r) }   }{ \sqrt{\mu(r)} }
&= - \ip{x} \cdot \frac{1}{ L^{(d+2)/2 } f(x) }  \ge - c \frac1{L^{d/2 +1} f(x)} .
\end{align*}
%%%%%%
%%%%% there it is
%Finally, in the case $r=L/2$, the same computations as above leave us with the estimate
%$ { \Delta f(x) }/{ f(x) }  \geq - d \cdot c L^{-1}  \ge - cr^{-1}, $
%for all $x \in S_{L/2}$.
%
%%%Altogether, whatever the value of $r$, if $x \in D_r$,
%%%%$$\tfrac{ \Delta f(x) }{ f(x) }  \ge - c r^{-1}.$$
Thus for $r \ge L/2$ and $k \ge 1$ we have, on $\cD_{r,k}$:
\begin{align*}
\sum_{x \in D_{r,k}} \tfrac{\Delta f(x)}{f(x)} L(t,x) & \ge - \tfrac{1/L^{d/2+1}}{f(x)} 2^{k+1} t f(x)^2  \times |D_r|  e^{ - c_{\ref{L:A}} 2^{k/2} r (\log (r+2) )^{-4}   } \\
& \ge - \tfrac{t}{L^{d/2 +1 }} L^{d-2}  \big( \sqrt{\mu_{L/2} } + \tfrac{r - L/2}{L^{d/2 +1 }} \big) \times  e^{ - c_{\ref{L:A}} 2^{k/2} r (\log (r+2) )^{-4}   }
\end{align*}
so that summing over $k$, on $\cD$, reasoning as above,
\begin{align*}
\sum_{x \in D_{r}} \tfrac{\Delta f(x)}{f(x)} L(t,x) & \ge  - c  \tfrac{t}{L^{d/2 +1 }} L^{d-2}  \big(  \tfrac1{L^{d/2} \log L} + \tfrac{r - L/2}{L^{d/2 +1 }} \big)
\end{align*}
and then summing over $r \ge L/2$:
\begin{align}
\sum_{r \ge L/2} \sum_{x \in D_{r}} \tfrac{\Delta f(x)}{f(x)} L(t,x) \ge - c ( \tfrac{t}{L^2 \log L} +  \tfrac{t}{L^2}) \ge - c \tfrac{t}{L^2}.  \label{Drcontr>}
\end{align}

\medskip \noindent \textbf{Step 4.} The contribution coming from $r= L/2$ is estimated in a similar way: for any $x \in S_{L/2}$,
$$
\tfrac{\Delta f(x)}{f(x)} \ge - c \tfrac{\log L}{L}
$$
and hence for the same reason as above, on the good event $\cS$,
\begin{align}
\sum_{x \in S_{L/2}} \tfrac{\Delta f(x)}{f(x)} L(t,x) & \ge - c L^{d-1} \tfrac{\log L}{L} \times t \tfrac{1}{L^d ( \log L)^2} \nonumber \\
& \ge - c \tfrac{t}{L^2 \log L}\label{contrmiddle}
\end{align}

\medskip \noindent \textbf{Conclusion.} Combining the results of all four steps above (\eqref{Srcontr}, \eqref{Drcontr<}, \eqref{Drcontr>}, and \eqref{contrmiddle}),
we deduce
$$
\sum_{x \in S} \tfrac{\Delta f(x)}{f(x)} L(t,x) \ge - c \tfrac{t}{L^2}.
$$
%{\color{blue} Need to finish contribution of $r=L/2$}
This concludes the proof of Lemma \ref{lem:RN bound}.
\end{proof}

\begin{remark}
Note that it is in \eqref{justifprofile} that we see the importance of the logarithmic correction terms in the choice of the local time profile $ \pi(x)$ in \eqref{eqn:2D trans prob}.
\end{remark}

With this lemma it is now easy to conclude the proof of the lower bound on the partition function.

\begin{proof}[Proof of Proposition \ref{P:lbZ}]
Let $d\ge 2$, and let $x = 0$ be the starting point of the walk.
Using the definition of $\cG$, Lemma \ref{lem:RN bound}, and \eqref{eq:G}, and the fact that $\pi(x) \le \pi(0) $ for any $x \in S$, we obtain:
\begin{align*}
Z(t,\beta) & = \E_0 [ \exp ( - \beta |\p R_t|)]  \ge \E_0 [ 1_{\cG} \exp ( - \beta |\p R_t|) ] \\
& \ge \exp ( -\beta c_{\eqref{cBoundary}}  L^{d-1}  ) \Pr_0( \cG)  = \exp ( -\beta c_{\eqref{cBoundary}}   L^{d-1}  )  \Q_0( 1_{\cG} \frac{d\Pr_0}{d\Q_0} )\\
& \ge \tfrac14  \exp ( -    (\beta  c_{\eqref{cBoundary}} L^{d-1} +  c_{\ref{lem:RN bound}}  t L^{-2}) ) .
\end{align*}
Recall that our choice of $L  = ( t  /\beta)^{1/(d+1)}$ guarantees that both terms
$\beta  L^{d-1}$ and $ t  L^{-2}$ in the exponential are of the same order of magnitude, namely
$ t^{1- 2/(d+1)} \beta^{2 / (d+1)}$.
%(This is saying that the entropic and energetic costs balance each other out).
This finishes the proof of Proposition \ref{P:lbZ} for a sufficiently large $\gamma$ (depending only on $\beta_0$ and the dimension $d$).
\end{proof}

\subsection{Discrete isoperimetry}

We now state and prove a modified isoperimetric inequality which deals with the outer boundary of a set.
We first need some definitions. For a set $G \subset \Z^d$, let $\text{Ext}(G)$ be the unique unbounded connected component of $\Z^d \setminus G$. Let the outer vertex boundary $\p ^* G$ be defined by
$$
\p^*G = \{ x \in G \ : \ \exists y \in \text{Ext}(G) \ , \ x \sim y\}.
$$
The outer edge boundary, denoted by $\p^*_e G$, consists of those edges $e = (x,y)$ with $x \in G$ and $ y \in \text{Ext}(G)$.

\begin{lem}
\label{lem:isoperimetric}
Let $A \subset \mathbb{Z}^d$ be a finite, connected set with $|A| \ge 2$.

\begin{enumerate}
\item Assume $d=2$, and let $R$ be the smallest rectangle in $\Z^2$ containing $A$ (i.e., $R$ is the intersection of all rectangles containing $A$).
Then,
\begin{equation}\label{iso2}
 |\p^* R| \leq 3 |\p^* A|.
 \end{equation}

\item For any $d \ge 2$,
\begin{equation}\label{iso3}
|\p^* A| \ge \frac{2d}{2d-1} |A|^{\frac{d-1}{d}}
\end{equation}
\end{enumerate}
\end{lem}

\begin{proof}
For any connected set $A$ such that $2 \le |A| < \infty$, we have that
\begin{equation}\label{edgevertex}
|\p^* A| \leq |\p^*_e A| \leq (2d-1) |\p^* A|.
\end{equation}
Indeed, for the first inequality simply note that the map which associates to an edge $e \in \p^*_e A$ the endpoint of $e$ which belongs to $A$ is a map from $\p^*_eA $ to $\p^*A$ which is clearly onto. This proves the first inequality. Moreover, any $x \in \p^*A$ has at most $2d-1$ pre-images in this map (since for any $x \in \p^* A$ then there are at most $2d-1$ edges in $\p^*_e A$ connected to $x$, and at least one other edge must connect $x$ to the rest of $A$, as $A$ is connected). This proves the second inequality and thus \eqref{edgevertex}.

Consider the case $d =2$. We claim that  $|\p^*_e R| \leq |\p^*_e A|$ (then we will see that \eqref{iso2} follows directly from \eqref{edgevertex}). Let $(x,x+e) \in \p^*_e A$, for some $e \in \set{ \pm e_i}$, where $e_i$ are the standard unit vectors of $\Z^2$, where
$x \in A$ and $x+e \notin A$. Note that since $x \in A$, we also have $x \in R$.  Thus, there exists a (necessarily unique) $k \geq 0$ such that $ y = x+k e \in R$ and
$  x + z e \not\in R$ for all $z > k$.  Thus, $(x+ k e, x+ (k+1)e) \in \p^*_e R$.
Hence we can define a map $\phi: \p^*_e A \to \p^*_e R$ by setting:
$$\phi((x,x+e)) = (y, y + e).$$
In words, we start from $(x, x+e)$ and travel in the direction $e$ until we leave $R$. This defines an edge in the outer edge boundary of $R$.

We claim that $\phi:\p^*_e A \to \p^*_e R$ is onto.
This follows since if
$(y,y+e) \in \p^*_e R$, then
considering the line $L = \set{ y - k e \ : \ k \geq 0 }$,
it must be that $L \cap A \neq \emptyset$, since otherwise
either $A$ would not be connected or $R$ would not be the smallest rectangle containing $A$.
(This relies on the assumption that $d=2$.)
Thus, there must exist some $k \geq 0$
such that $y - k e \in A$ and $y - z e \not\in A$ for any $z < k$.
Thus, the edge $(y - k e, y - (k-1)e)$ is in $\p^*_e A$,
and it is immediate that $\phi(y-ke, y-(k-1)e) = (y,y+e)$.

This proves that there is a map from $\p^*_e A$ onto $\p^*_e R$, and hence $|\p^*_e R | \le |\p^*_e A|$. Therefore, by \eqref{edgevertex},
$|\p^* R| \leq |\p^*_e R| \leq |\p^*_e A| \leq 3|\p^* A|$, which proves \eqref{iso2}.

For the general case $d\ge 3$ we use the discrete Loomis--Whitney inequality (Theorem 2 in \cite{LoomisWhitney}), which states that if $A_i$ is the projection of $A$ onto $\Z^{d-1}$ along the $i$th coordinate then
\begin{equation}\label{LW}
|A|^{d-1} \le \prod_{i=1}^d |A_i|.
\end{equation}
For each $1\le i \le d$ and each vertex in $z \in A_i$ consider the line $L$ going through $z$ and which is parallel to the $i$th coordinate axis. It intersects $A$ in at least one vertex (assume for simplicity and without loss of generality that $A$ does not intersect any hyperplane where one of the coordinates is 0). The first and last such intersections with $A$ necessarily correspond to two edges in $ \p_e^* A$, since the rest of the line lies in $\Z^d \setminus A$ and is unbounded. Thus to each vertex in $A_i$ one can associate two edges in $ \p^*_e A$. Note that for two distinct vertices $z$ and $w$ the corresponding edges will be pairwise distinct. Hence $|A_i| \le | \p^*_e A | / 2$ for each $1\le i \le d$.  We deduce, using the arithmetic geometric inequality and \eqref{edgevertex},
$$
\prod_{i=1}^d |A_i| \le  \left( \frac1d \sum_{i=1}^d |A_i| \right)^{d} \le \left(\frac1{2d} |\p^*_e A|\right)^{d}  \le \left(\frac{2d-1}{2d} |\p^* A|\right)^{d}.
$$
Combining with \eqref{LW} this gives the desired result.
\end{proof}

\subsection{Proof of condensation}

\label{S:proofs}

We will prove the following more precise statement of Theorem \ref{T:2dimGibbs}.

\begin{thm}
  \label{T:2dimGibbs+}
  Let $d \ge 2$.
  Fix $\beta_0>0$ and let $\beta > \beta_0 $.
Let $\gamma$ be as in Proposition \ref{P:lbZ}.
Then,
  \begin{equation}
\label{Teq:2d+cond}
\mu_t \Big[ \diam(R_t) \geq \frac1{\sqrt{2\gamma}}   \left( \tfrac{t}{\beta}\right)^{1/(d+1)}
 \Big] \geq 1 - C \exp \sr{ - \gamma t^{1 - 2/(d+1)}  \beta^{2/(d+1)} }  ,
  \end{equation}
and if $d = 2$ then
$$ \mu_t \Big[  \diam(R_t) \leq 6\gamma \left( \tfrac{t}{\beta}\right)^{1/3}
 \Big] \geq 1 - C \exp \sr{ - \gamma t^{1/3}  \beta^{2/3} } . $$

 Moreover, for all $d \geq 2$,
 $$ \mu_t \Big[ |R_t| \leq (2\gamma)^{d/(d-1)} \sr{ \tfrac{t}{\beta} }^{d/(d+1) } \Big] \geq 1 - C \exp \sr{ - \gamma t^{1 - 2/(d+1)}  \beta^{2/(d+1)} } . $$
\end{thm}

\begin{proof}
Recall that by Proposition \ref{P:lbZ}
\begin{equation}
Z(t, \beta) \ge  \exp\left(- \gamma t^{1- 2/(d+1)} \beta^{2 / (d+1)} \right) .
\label{Zlowerbound}
\end{equation}

We start with the lower bound on the diameter. We require the following standard estimate. Let $R^\square_t$ denote the smallest $d$-dimensional box containing $R_t$. For $1\le i \le d$, let $J^i_t$ denote the length of the projection of $R^\square_t$ (or equivalently $R_t$) onto the $i$th coordinate axis.

\begin{lem}
  \label{L:stayinsquare}
We have
  $$
  \Pr_0 [  J_t^i \le n ] \le n  \exp \sr{  - t \frac{\pi^2}{2n^2}  }
  $$
\end{lem}

\begin{proof}
  Under $\Pr_0$, the coordinates $X^1_t,  \ldots, X^d_t$ are independent continuous time (with rate $2$) simple random walks on $\Z$. We just focus on the first coordinate, $X_t = X^1_t$, and compute
  $
  \Pr_x(T > t)
  $
  where $x \in \{1, \ldots, J\}$ and $T = \inf\{t \ge 0: X_t \notin [1,J-1]\}$.
  Let $\cL$ denote the generator of (rate 1) simple random walk on $\Z$, and let $\phi(x) = e^{i \pi x / J}$. It is trivial to check  that
  $$
  \cL \phi(x) = - \lambda \phi(x)
  $$
  for all $x \in \Z$, where $\lambda = 2(1- \cos(\pi/J))$. Thus if we let $\psi(t,x) = e^{\lambda t} \sin(\pi x / J)$ we have
  $$
  \frac{\p}{\p t} \psi + \cL \psi = 0
  $$
  and hence $M_t: = e^{\lambda t} \sin(\pi X_t / J)$ is a martingale. Consequently, applying the optional stopping time theorem at the time $t \wedge T$ (which is bounded), and the inequality $\sin(u) \ge (2/\pi) u$ valid for $0\le u \le \pi/2$, yields
  \begin{align*}
  \sin(\pi x/J) &= \E_x( e^{\lambda  t} \sin(\pi X_{t \wedge T}/J))\\
  & \ge e^{\lambda t} \frac{2}{J} \Pr_x( T > t).
  \end{align*}
  Therefore,
  $$
  \Pr_x(T> t) \le J e^{-\lambda t}.
  $$
  Now, $\lambda = 2(1- \cos(\pi /J)) \ge \pi^2/{(2 J^2)}$ for $J$ large enough, and the result follows.
\end{proof}

We now deduce from Lemma \ref{L:stayinsquare} a lower bound on the diameter of $R_t$ under $\mu_t$.
Let $J^1_t, \ldots,J^d_t$ be the side-lengths of $R^{\square}_t$. Let $N = \min \{ J^1_t, \ldots, J^d_t \}$. We will prove the stronger statement that $N \ge c (t/\beta)^{1/(d+1)}$ with high probability.
%This will also be needed in the upper bound in dimension $2$.

By Lemma \ref{L:stayinsquare} and Proposition \ref{P:lbZ},
for an integer $n>0$,
\begin{align*}
\mu_t [ J^1_t \leq n ]
& \leq Z(t,\beta)^{-1} n C \exp \sr{ - \frac{\pi^2 t}{ 4 d n^2 } } \\
&\leq C n \exp \sr{ \gamma  t^{1- 2/(d+1)} \beta^{2 / (d+1)}  - \frac{\pi^2}{2n^2} t }  .
\end{align*}
Thus, if
$ n = (1/\sqrt{2\gamma} )( t/\beta)^{1/(d+1)} $ and since $\pi^2\ge 2$,
we get that
$$ \mu_t [ J^1_t  \leq n ] \leq Cn \exp \sr{ - \gamma t^{1-2/(d+1)} \beta^{2/(d+1)}  } . $$
Of course, we get the same bound replacing $J^1_t$ by $J^i_t$. Therefore,
\begin{equation}\label{lbN}
\mu_t [ N\le n ] \le O(t/\beta)^{d+1} \exp \sr{ - \gamma t^{1-2/(d+1)} \beta^{2/(d+1)}  }.
\end{equation}
In particular, it holds that with high $\mu_t$-probability
$$
\diam(R_t) \geq \frac1{\sqrt{2\gamma}} (t/\beta)^{1/(d+1)}  . $$

We now turn to the upper bound on the diameter in dimension $d=2$.
We make the following observation. In dimension $d=2$, if we know that the diameter of a shape $G$ is $\ge M$ for some large $M$ then we will see that it automatically follows (by Lemma \ref{lem:isoperimetric}) that $|\p G | \ge c M$. This ensures that the energy associated to this particular shape is at least $c \beta M$. This is enough for proving the theorem in the $d=2$ case. [On the other hand, in dimension 3 and higher, such a simple relationship is no longer true: if $\diam(G) \ge M$ then we can only infer that $| \p G | \ge c M$, translating into an energy cost of $c \beta M$. This is far less than what we need, since we believe the relevant energy contributions are of order $\beta M^{d-1}$. The issue is that a shape could have a big diameter in one direction and be very ``thin" along  other directions.]

More precisely, recall that $R_t^{\square}$ is a $J_t^1 \times J_t^2$ rectangle.
Lemma \ref{lem:isoperimetric} tells us that
$| \p R_t| \geq |\p^* R_t | \geq \tfrac13 |\p^* R_t^{\square} | = \tfrac23 \cdot (J_t^1 + J_t^2)$.
Thus,
\begin{align*}
\mu_t [ J_t^i > m ] & \leq Z(t,\beta)^{-1} \sum_{k=m+1}^\infty \exp ( - \beta \tfrac23 k ) \\
& \leq C \exp ( - \beta \tfrac23 m + \gamma t^{1/3} \beta^{2/3} )   .
\end{align*}
If $m = 3\gamma (t /\beta)^{1/3}$ this probability is at most
$C\exp ( - \gamma t^{1/3} \beta^{2/3} )$.  A union bound over $i=1,2$
give that in particular, $\diam (R_t) \le 2 m$ with high probability,
which concludes the proof of the first part of Theorem \ref{T:2dimGibbs+}.

We turn to the second part of the proof which yields an upper bound on the volume of $R_t$ in all dimensions $d\ge 2$. For this we note that by Lemma \ref{lem:isoperimetric}, if $|R_t| \ge m$ then $|\p R_t| \ge |\p^* R_t | \ge (2d/ (2d-1)) m^{(d-1)/d} \ge m^{(d-1)/d}$, and so almost surely on this event, $ \exp( - \beta H(\omega)) \le \exp ( -  \beta m^{(d-1)/d} )$.
Consequently,
\begin{align*}
\mu_t [ | R_t | \geq m ] & \le Z(t,\beta)^{-1} \exp ( - \beta m^{(d-1)/d } ) \cdot \Pr [ |R_t| \geq m ] \\
& \leq \exp \sr{ - \beta m^{(d-1)/d } + \gamma t^{1-2/(d+1)} \beta^{2/(d+1)} } .
\end{align*}
If $m^{(d-1)/d} = 2\gamma ( \tfrac{t}{\beta} )^{(1 - 2/(d+1) ) }$,
or equivalently, $m =  (2\gamma)^{d/(d-1)}( \tfrac{t}{\beta})^{ d / (d+1) }$, this probability is at most
$\exp ( - \gamma t^{1-2/(d+1) } \beta^{2/(d+1)} )$.
This completes the upper bound on the volume in all dimensions and thus the proof of the theorem.
\end{proof}

\subsection{Proof of Theorems \ref{Tvariant} and \ref{T:Cond2dim}}

We explain how to adapt the arguments of the proof of Theorem \ref{T:2dimGibbs} to give the proof of Theorem \ref{T:Cond2dim}.
Let $K>0$ be large enough and let $S' = \cup_{r > K} S_r$. Let $\cB' = \{ \forall x \in S': L(t, x) \ge \beta\}$. Let $ \cG = \cG'_t = \cB' \cap \cS \cap \cD$. Then the same arguments as in \eqref{eq:G}  show that
$
\Q(\cG'_t) \ge 1/4
$, provided that $K$ is a sufficiently large constant. The only difference with \eqref{eq:G} is that it no longer suffices to bound the expected number of vertices that were not visited by time $t$ as in Lemma \ref{L:ER}, which followed directly from Proposition \ref{P:exphit}. Instead, we need to show that the local time at every vertex in $S'$ is greater than $\beta$ with probability greater than $1/2$ say. However this is a direct consequence of the lower bound large deviations discussed in Remark \ref{R:LTconc_low}.

We deduce that
$$
\Pr (\cG') \ge \exp ( - \gamma t^{1-2/(d+1)} \beta^{2 /(d+1)})
$$
for some large enough constant $\gamma$ depending only on $\beta_0$ and $d$.
Assume that $\cG'_{t}$ holds. In the next $t$ units of time, we make sure that the each of the remaining $O(KL^{d-1})$ vertices of $S\setminus S'$ are visited at least $\beta$ units of time each, as follows. For each $1\le k \le K$, we visit each vertex in $S_k$ in clockwise order, starting from $( k, 0, \ldots, 0)$. At each new vertex, the walk remains at least $\beta$ and at most $2\beta$ units of time. When the walk has visited each vertex of $S_k$, it moves on to $S_{k+1}$. The total amount of time spent doing so is at most $2 \beta K L^{d-1} \le  2K t /L^2 $, which is much less than the $t$ units of time in which we want to achieve this, since by assumption $\beta =o( t) $. In the remaining amount of time, the walk is free to do what it wants, provided it stays in $S$.

If all these conditions are fulfilled, it is clear that $R_{2t} = S$ and that each vertex has a local time greater than $\beta$, so $\cE_{2t}$ holds. The probability of visiting every vertex in this  prescribed order immediately after $t$ is at least $\exp( - c \beta K L^{d-1}  ) $ for some $c < \infty$. The probability of remaining in $S$ after that (for a time necessarily shorter than $t$) is easily seen to be at least $\exp ( - ct /L^2 )$ and hence at least $\exp ( - c t^{1-2/(d+1)} \beta^{2/(d+1)})$.  All in all, we deduce
\begin{equation}
\label{lbEt}
\Pr (\cE_{2t}) \ge \exp ( - \gamma  t^{1- 2/(d+1)} \beta^{2/(d+1)}),
\end{equation}
and thus (changing $t$ into $t/2$) the same inequality holds with the left hand side replaced by $ \Pr(\cE_t)$.
This argument also shows that if $\tilde Z(t, \beta)$ is the partition function corresponding to the Hamiltonian $\tilde H = \sum_{x \in \p R_t} L(t,x)$ in \eqref{Hvariant}, then
\begin{equation}\label{tildeZlb}
\tilde Z(t, \beta) \ge \exp ( - \gamma t^{1 - 2/(d+1)} \beta^{2/(d+1)}).
\end{equation}
(In fact, this could also be deduced from Corollary \ref{cor:tail}.)

Now, we claim that for any finite set $G$ of vertices,
\begin{equation}\label{expcostEt}
\Pr(R_t = G, \cE_t) \le \exp( - \beta  |\p G|).
\end{equation}
For each $x \in \p G$, let $y$ be a neighbour of $x$ such that $y \notin G$. Consider the event $J_{xy}(t)$ that by time $t$ there has never been a jump from $x$ to $y$.
On $\cE_t$, $x$ is visited at least $\beta$ units of time. While at $x$, the rate of jumping to $y$ is of course $1$. Let $E_{xy}$ be independent exponential random variables with rate $1$, which represents the amount of time a particle would have to wait before jumping to $y$. Thus $J_{xy}(t) \cap \cE_t \subset \{E_{xy} > \beta\}$. Hence
$$
\Pr(R_t = G, \cE_t) \le \Pr ( \cap_{x \in \p G} J_{xy}(t) \cap \cE_t) \le \Pr( \cap_{x \in \p G} E_{xy} > \beta)   \le e^{- \beta  |\p G| }
$$
by independence of the random variables $E_{xy}$. Thus \eqref{expcostEt} is established.

Putting together \eqref{lbEt} and \eqref{expcostEt} (resp. \eqref{tildeZlb} and the definition of $\tilde \mu$), the proof of Theorem \ref{T:Cond2dim} (resp. Theorem \ref{Tvariant}) proceeds essentially as in Theorem \ref{T:2dimGibbs+}. More precisely, let $J_t^1, \ldots, J_t^d$ be the dimensions of $R_t$ in each coordinate. The lower bound in \eqref{lbEt} implies exactly as in \eqref{lbN} that
$$
\Pr( \min(J_t^1, \ldots, J_t^d) \ge n | \cE_t) \to 1
$$
as $t \to \infty$, where $n = (1/\sqrt{2\gamma})(t/\beta)^{1/(d+1)}$. In particular, conditioned on $\cE_t$, with high probability we have $\diam (R_t) \ge n$.

For the upper-bound on $\diam(R_t)$ in the case $d=2$, or the upper bound on $|R_t|$ in the general case $d\ge 2$, we proceed as follows. We focus on the bound on $|R_t|$ in the general case $d\ge2$, which requires a few more ideas. For each edge $e$, consider the unit area plaquette $p(e)$, orthogonal to $e$ and such that the centre of $p(e)$ coincides with the midpoint of the edge $e$.
\begin{dfn}
By a self-avoiding surface, we mean a connected union of plaquettes with disjoint $(d-1)$-dimensional interior.
\end{dfn}
When $d=2$, this is essentially equivalent to a self-avoiding walk. Let $\cS_{n}$ denote the set of self-avoiding surfaces with $n$ plaquettes, and contained in a ball of radius $n$ about the origin. Let $c_n = |\cS_n|$
and let
\begin{equation}\label{beta0}
\alpha = \alpha(d) = \limsup_{n \to \infty} c_n^{1/n} \qquad \beta_0 = \log \alpha.
\end{equation}
Note that when $d = 2$, the limsup is a limit and is (essentially by definition) equal to the connective constant of $\Z^2$. It is easy to check that $1\le \alpha \le (2d)^{2d} < \infty$ in general, which is all we will use.

To each finite $G \subset \Z^d$ we can associate a finite self-avoiding surface, where the plaquettes are obtained by considering each of the edges $e = (x,y)$, with $x \in G$ and $y \in \Ext(G)$.
Let $\cS_{j_1, \ldots,j_d}$ denote the set of surfaces where the diameter in each direction $1, \ldots, d$, does not exceed $j_1, \ldots, j_d$ respectively.
Let $\Sigma$ be the (random) self-avoiding surface associated with $R_t$. For a given self-avoiding surface $\sigma \in \cS_j$, we have by the same argument as in \eqref{expcostEt} (since each plaquette corresponds to an edge $(x,y)$ such that the corresponding exponential random variable $E_{xy}$ satisfies $E_{xy} > \beta$, and these events are independent even for edges which share vertices)
\begin{equation}\label{expcostGamma}
\Pr ( \Sigma = \sigma , \cE_t ) \le \exp ( -  \beta j  ),
\end{equation}
Let $\beta_1 > \beta_0 = \log \alpha$ and assume that $\beta>\beta_1$.  Let $\beta'_1  = (\beta_0 + \beta_1 )/2$. Note that for $n$ large enough, we have $|\cS_n| \le \exp ( \beta_1' n)$.

Therefore, by \eqref{expcostGamma},
\begin{align*}
\Pr [\Sigma \in \cS_n | \cE_t] & \leq  \Pr(\cE_t)^{-1} \sum_{j =n}^{\infty}  e^{ \beta'_1j } e^{ - \beta j}   \\
& \le C \exp \{ \gamma t^{1  - 2/(d+1) } \beta^{2/(d+1)}
 -  ( \beta - \beta'_1)  n \}
\end{align*}
where $C = \sum_{j\ge 0} \exp(- (\beta - \beta'_1) j) \le \sum_{j} \exp ( - j (\beta_1 - \beta'_1) /2) < \infty$ since $\beta_1 > \beta'_1$.
Let
$$n = \left\lceil \frac{\gamma t^{1 - 2/(d+1) } \beta^{2/(d+1)}} {2(\beta - \beta'_1)} \right\rceil \le C \gamma (t/\beta)^{\frac{d-1}{d+1}}$$
where $C$ depends only on $\beta_1$. Then we deduce
$$
\Pr [\Sigma \in \cS_n | \cE_t ] \to 0,
$$
Hence $|\p^* R^\square_t| \le n$ with high conditional probability given $\cE_t$, and thus (by Lemma \ref{lem:isoperimetric})
$$
|R_t| \le   [(2d -1)  n]^{d/ (d-1)} \le C \gamma (t/\beta)^{\frac{d}{d+1}}
$$
with high conditional probability, as desired.

\begin{remark}
It is interesting to note that the lower bound on $\diam(R_t)$ is valid for all $\beta>0$ (i.e., does not assume $\beta> \beta_0$).
\end{remark}

\section{Open problems and conjectures}
\label{S:opb}
We finish the paper with a brief discussion of some open problems raised by our results.

\paragraph{Limit shape theorem.} The most basic question is to ask whether the constants $c_1$ and $c_2$ appearing in Theorem \ref{T:2dimGibbs} really need to be different from one another, and if indeed $t^{1/(d+1)}$ is the right order of magnitude in all dimensions $d\ge 2$. We make the following more precise conjecture:

%\begin{opb}Does there exist a constant $c $ such that $\diam(R_t) / t^{1/(d+1)}$ converges to $c$ in $\mu_t$-probability as $t\to \infty$?
%\end{opb}

%Another immediate question is to obtain the limiting shape of the random walk. Suppose $d =2$ for simplicity.

%\begin{opb}\label{conj:square}
\medskip \noindent \textbf{Conjecture.} There exists a nonrandom closed, bounded and convex set $S = S(\beta) \subset \R^d$ such that
$$
\inf_{z\in \R^d}d_{\Haus} (\frac{R_t}{\diam(R_t)} ; z+ S) \to 0
$$
in probability, where $d_{\Haus} $ stands for Hausdorff distance.
%\end{opb}

An equivalent way of stating the conjecture is that there exists a deterministic $S$ (compact and convex) such that if we translate the range $R_t$ to have a centre of mass at the origin, then the resulting set is close to $S$ with high probability in the Hausdorff sense. This is similar to the situation in \cite{bolthausen}.

\medskip Once the existence of $S$ is established one may ask numerous questions about its geometry. For instance, does it have any (macroscopic) flat facet?

\medskip Studying the extreme cases $\beta \to \infty$ and $\beta \to 0$ should also be interesting. Further to the above open problem, we conjecture that as $\beta \to \infty$, $S(\beta)/\diam(S(\beta))$ converges in the Hausdorff sense to a diamond of unit diameter. This is because the diamond is the minimiser of the isoperimetric problem for the vertex-boundary:
$
\min_{|S| = k} |\pd S|
$
is attained for a diamond $S = \{x, y : |x| + |y| \le n\}$, whenever $k = 2n(n+1)$. Since this conjecture was first made, a very closely related result has been proved by Biskup and Procaccia \cite{BP1, BP2}. At the other extreme, as $\beta \to 0$ it is natural to believe that the lattice effects become less and less relevant, so that the limit shape becomes rotationally invariant. Thus we conjecture that $S(\beta) / \diam(S(\beta))$ converges as $\beta \to 0$ in the Hausdorff sense to a ball of unit diameter. This seems intuitively related to the result of Duminil--Copin on the limit of the Wulff crystal for percolation as $p \to p_c$ on the triangular lattice (\cite{DuminilCopin}).

\medskip We make similar conjectures for the case of a random walk conditioned on $\{ L_t(x) \ge \beta, \forall x \in R_t\}$. However, in the case $\beta \to \infty$ we believe that the limit should be a square with unit diameter instead of a diamond. This is because by \eqref{expcostEt}
$$
\Pr[ R_t = G , \Ee_t ] \leq \exp ( - \beta | \p_e G | )
$$
where $\p_e G$ denotes the edge boundary of a graph $G$. Thus, when $\beta\to \infty$, it is reasonable to guess that $S(\beta)$ should minimise its edge boundary, rather than its vertex boundary, and hence be a square rather than a diamond. As we do not yet know whether the behaviour described in Theorem \ref{T:Cond2dim} persists for $\beta \le \beta_0$, we do not make any conjecture for the case $\beta \to 0$.

\paragraph{Fluctuations.} The question of the roughness of the boundary of the shape is of considerable interest.
In the case of two-dimensional percolation, these fluctuations are known with considerable precision. For instance (see \cite{UzunAlexander} and \cite{Alexander}), the \emph{maximal local roughness}, which measures the maximal distance from a point on the boundary of the shape to the polygonal hull of that shape (and hence the size of inward deviations),
 is of order $(\text{diameter})^{1/3}$. More recently, Hammond \cite{Hammond1} established an extremely precise result in this direction which gives a sharp logarithmic power-law correction (stated in the greater generality of the $q$-state Potts model with $q\ge 1$). This exponent and related ones are common to a large class of two-dimensional interfaces, including the KPZ (Kardar--Parisi--Zhang) universality class. We conjecture that this is the case here as well; and since the diameter itself is of order $t^{1/3}$, this leads us to the following:

 \medskip \noindent \textbf{Conjecture.} For any $\beta >0$, with high $\mu_t$-probability, the maximum local roughness of $R_t$ is of order $t^{1/9}$, up to logarithmic corrections.

%We speculate that in two dimensions, the exponent for the fluctuations take the value in agreement with the Wulff crystal (and hence also with KPZ): this suggests that the maximum facet length is of order $t^{2/9}$ and the maximum local roughness of order $t^{1/9}$.

%It would be interesting to know whether for some small values of $\beta$ the behavior can be markedly different to the case where $\beta > \beta_0$. This appears to be a rather delicate question.

% -----------------------------------------------------------------------------------
% ---------------------------- BIBLIOGRAPHY -----------------------------------------
% -----------------------------------------------------------------------------------

%\bibliographystyle{AYbibstyle}
%\bibliography{mybib}

\end{document}